\documentclass[a4paper,11pt]{amsart} 
\usepackage{amsfonts, amsmath,amssymb,amsbsy,amstext,amsxtra,latexsym}
\usepackage[all]{xy}

\newtheorem{thm}{Theorem}[section]

\newtheorem{theorem}[thm]{Theorem}

\newtheorem{lemma}[thm]{Lemma}
\newtheorem{prop}[thm]{Proposition}
\newtheorem{cor}[thm]{Corollary}

\newtheorem{rmk}[thm]{Remark}
\newtheorem{exm}[thm]{Example}

\title[Stability of nets of quadrics in $\mathbb{P}^5$]
{Stability of nets of quadrics in $\mathbb{P}^5$ and associated discriminants}
\author{Sangho Byun}
\date{}

\address{Department of Mathematical Sciences, KAIST, 291 Daehak-ro, Yuseong-gu, Daejeon 305-701, Korea}

\email{capqus@kaist.ac.kr}

\subjclass[]{}

\begin{document}
\maketitle

\begin{abstract} Let $S$ be a complete intersection surface defined by a net $\Lambda$ of quadrics in $\mathbb P^5$. In this paper we analyze GIT stability of nets of quadrics in $\mathbb P^5$ up to projective equivalence, and discuss some connections between a net of quadrics and the associated discriminant sextic curve. In particular, we prove that if $S$ is normal and the discriminant $\Delta(S)$ of $S$ is stable then $\Lambda$ is stable. And we prove that if $S$ has the reduced discriminant and $\Delta(S)$ is stable then $\Lambda$ is stable. Moreover, we prove that if $S$ has simple singularities then $\Delta(S)$ has simple singularities.
\end{abstract}

\section{Introduction}

One of central theme in algebraic geometry is to construct compact moduli
spaces with geometric meaning. There are two early successes of the moduli theory - the
construction and compactification of the moduli spaces of curves $\overline{\mathcal{M}}_g$ and principally
polarized abelian varieties (ppavs) $\overline{\mathcal{A}}_g$. While
very few other examples are so well understood.

One of the oldest approach to moduli problems is Geometric Invariant Theory (GIT). The GIT stability analysis for pencils of quadrics appear in \cite{AM},\cite{AL} and \cite{MM}, it is shown that a pencil of quadrics in $\mathbb{P}^n$ is stable (resp. semistable) if and only if the associated discriminant binary $(n+1)$-form is non-zero and is GIT stable (resp. semistable) with respect to the natural $SL(2)$-action. So the GIT stability of a pencil of quadrics can be read off the GIT stability of the associated discriminant locus. But the GIT analysis for nets of quadrics turns out to be more involved. In particular, as Example~\ref{exm1} shows, there is no natural correspondence between $SL(6)$-stability of a net and $SL(3)$-stability of the associated discriminant curve. Moreover, the complete analysis of stable locus is complicated. For example, see \cite{FS} for discussion of GIT stability of net of quadrics in $\mathbb{P}^4$. However, we know the following facts: if a net defines a complete intersection surface with simple singularities then the net is stable with respect to $SL(6)$-action (\cite{OSS},\cite{Odaka1},\cite{Odaka2} and \cite{LiTian}). And if a net defines a smooth complete intersection surface, then the associated discriminant curve is stable with respect to $SL(3)$-action \cite{Be2}.

Firstly, we find GIT stability criterion of net $\Lambda$ of quadrics in $\mathbb{P}^5$ via studying special one-parameter subgroups. Our GIT analysis follows the method in Section 2 of \cite{FS}.

\begin{theorem}
(=Theorem ~\ref{1-PS})
Suppose that $\Lambda$ is stable with respect to every one-parameter subgroup of the following numerical types:
\begin{align*}
&{\rm(1)} \ \rho_1=(1,1,1,1,1,-5). && {\rm(8)} \ \rho_8=(4,1,1,-2,-2,-2). \\
&{\rm(2)} \ \rho_2=(1,1,1,1,-2,-2). && {\rm(9)} \ \rho_9=(3,1,1,-1,-1,-3). \\
&{\rm(3)} \ \rho_3=(1,1,1,-1,-1,-1). && {\rm(10)} \ \rho_{10}=(2,1,0,0,-1,-2). \\
&{\rm(4)} \ \rho_4=(2,2,-1,-1,-1,-1). && {\rm(11)} \ \rho_{11}=(5,5,-1,-1,-1,-7). \\
&{\rm(5)} \ \rho_5=(5,-1,-1,-1,-1,-1). && {\rm(12)} \ \rho_{12}=(1,1,0,0,-1,-1). \\
&{\rm(6)} \ \rho_6=(2,2,2,-1,-1,-4). && {\rm(13)} \ \rho_{13}=(5,3,1,-1,-3,-5). \\
&{\rm(7)} \ \rho_7=(7,1,1,1,-5,-5).
\end{align*}
Then $\Lambda$ is stable.
\end{theorem}

On the basis of this partial analysis, we may already conclude the important fact that a stable net has a pure two-dimensional intersection, and hence defines a connected surface with local complete intersection singularities.

\begin{cor}
(=Theorem ~\ref{ci})
If a net of quadrics in $\mathbb{P}^5$ is stable, then the corresponding intersection is connected and purely two-dimensional.
\end{cor}

And our analysis makes us possible to discuss some connections between a net of quadrics and the associated discriminant sextic curve.

More precisely, we prove the following theorems.

\begin{theorem}
(=Theorem ~\ref{normal})
Suppose a net $\Lambda$ of quadrics in $\mathbb{P}^5$ defines a complete intersection normal surface $S$. If the discriminant $\Delta(S)$ of $S$ is stable then $\Lambda$ is stable.
\end{theorem}

\begin{theorem}
(=Theorem ~\ref{reduce})
Suppose a net $\Lambda$ of quadrics in $\mathbb{P}^5$ defines a complete intersection surface $S$ with the reduced discriminant $\Delta(S)$. If $\Delta(S)$ is stable then $\Lambda$ is stable.
\end{theorem}

Let $S$ be a K3 surface of degree $8$ in $\mathbb{P}^5$, given by the complete intersection of three quadrics. Associated to $S$ is a K3 surface $S'$ which is a double cover of $\mathbb{P}^2$ ramified over a sextic. And there is a dominant rational map $\phi: \mathcal{N}^{ss}//SL(6)\dashrightarrow \mathcal{C}_6^{ss}//SL(3)$ where $\mathcal{N}$ the space of nets of quadrics in $\mathbb{P}^5$ and $\mathcal{C}_6$ the space of plane sextic curves. The degree of this map is equal to the number of non-effective theta characteristics on a general sextic curve which is $2^9(2^{10}+1)$ (cf. \cite{Do} and \cite{Sh}). Our original motivation in this paper is to compare the moduli space of algebraic K3 surfaces with degree $8$ with the moduli space of K3 surfaces as a double cover of $\mathbb{P}^2$, ramified over a sextic curve. It is well known that if $S$ is nonsingular then $\Delta(S)$ is stable \cite{Be2}.

\begin{theorem}\label{int-simple}
(=Theorem ~\ref{simple})
Suppose a net $\Lambda$ of quadrics in $\mathbb{P}^5$ defines a complete intersection surface $S$. If $S$ has simple singularities, then $\Delta(S)$ has simple singularities.
\end{theorem}

The morphism from the moduli space of degree $8$ K3 surfaces to degree $2$ K3 surfaces has been studied from many points of view, starting from Mukai's paper \cite{Mukai}. His result implies directly that if $S$ is a smooth K3 surface which is a complete intersection of three quadrics in $\mathbb{P}^5$ then $\Delta(S)$ has simple singularities.

It has an interpretation in Hodge-theoretic terms that should yield Theorem~\ref{int-simple} fairly quickly. One can compare the Hodge structures on the two K3 surfaces and interprets what it means for their images to have simple singularities in terms of (-2)-class in the Picard group \cite{Hassett}. Our approach is rather direct via GIT analysis even though it involves complicated computations.

In Section 2, we describe a finite set of numerical types of one parameter subgroups $\{\rho_i\}_{i=1}^N$ such that the union of the $\rho_i$-nonstable points is $Gr(3,W)\backslash Gr(3,W)^{s}$. In Section 3, we use this result to prove our main theorems.
In this paper, we work on the field of complex numbers.

\section{Stability of nets of quadrics}

Our notations and GIT analysis follow Section 2 of \cite{FS}.

Let $V=H^0(\mathbb{P}^5,\mathcal{O}(1))$ and let $W=H^0(\mathbb{P}^5,\mathcal{O}(2))$ be the space of quadratic forms. A net of quadrics in $\mathbb{P}^5$ is by definition a plane in $\mathbb{P}(\bigwedge ^3W)$. So the space of nets of quadrics is by definition the Grassmannian $Gr(3,W)\subset \mathbb{P}(\bigwedge ^3W)$.
%
%

Let $\rho=(a_0,a_1,a_2,a_3,a_4,a_5):\mathbb{C}^*\rightarrow SL(6)$ be a normalized one-parameter subgroup(1-PS), i.e. $\rho$ is a one-parameter subgroup, acting diagonally on a basis $\{x_0,x_1,\dots,x_5\}$ of $V$ with weights $\{a_0,a_1,a_2,a_3,a_4,a_5\}$ satisfying $a_0\ge a_1\ge a_2\ge a_3\ge a_4\ge a_5$ and $\sum_{k=0}^{5} a_k=0$.

Then the $\rho$-weight of a quadratic monomial $x_ix_j$ is
$$
\omega_{\rho}(x_ix_j)=a_i+a_j
$$
and the $\rho$-weight of the Pl\"{u}cker coordinate $x_{i_1}x_{j_1}\wedge x_{i_2}x_{j_2}\wedge x_{i_3}x_{j_3}$ is simply $\sum_{k=1}^3 \omega_{\rho}(x_{i_k}x_{j_k})$.

By the Hilbert-Mumford numerical criterion (Theorem 2.1 in \cite{GIT}), a net $\Lambda$ is stable (resp., semistable) with respect to $\rho$ if there exists a Pl\"{u}cker coordinate that does not vanish on $\Lambda$ with positive (resp., non-negative) $\rho$-weight. And $\Lambda$ is stable (resp., semistable) if and only if $\Lambda$ is stable (resp., semistable) with respect to all one-parameter subgroups.

A priori, the numerical criterion requires one to check $\rho$-stability for all one-parameter subgroups. However, there necessarily exists a finite set of numerical types of one-parameter subgroups $\{\rho_i\}_{i=1}^N$ such that the union of the $\rho_i$-nonstable points is $Gr(3,W)\backslash Gr(3,W)^{s}$. The first main result of this section, Theorem~\ref{1-PS}, describes such a set of one-parameter subgroups explicitly.

Throughout this section, we use the following notations. Given a basis $\{x_0,x_1,\dots,x_5\}$ of $V$ and a normalized 1-PS $\rho$ acting on $\{x_0,x_1,\dots,x_5\}$, we can define two complete orderings on the set of quadratic monomials:

(1) The lexicographic ordering ``$\succ_{lex}$",

(2) ``$\succ_{\rho}$": $m_1\succ_{\rho} m_2$ if either $\omega_{\rho}(m_1)> \omega_{\rho}(m_2)$ or $\omega_{\rho}(m_1)= \omega_{\rho}(m_2)$ and $m_1\succ_{lex} m_2$.

And there is another ordering ``$\geqslant$", according to which $m_1\geqslant m_2$ if and only if $\omega_{\rho}(m_1)\ge \omega_{\rho}(m_2)$ for any normalized 1-PS acting diagonally on $\{x_0,x_1,\dots,x_5\}$. Note that $m_1\geqslant m_2$ implies $m_1\succeq_{lex} m_2$.

For any quadric $Q\in W$, we denote by $in_{lex}(Q)$ the initial monomial of $Q$ with respect to $\succ_{lex}$ and if $\rho$ is normalized 1-PS acting on $\{x_0,x_1,\dots,x_5\}$, we denote $in_{\rho}(Q)$ by the initial monomial of $Q$ with respect to $\succ_{\rho}$.

For any net $\Lambda=(Q_1,Q_2,Q_3)$, we can always choose a basis $(Q'_1,Q'_2,Q'_3)$ with $in_{lex}(Q'_1)\succ_{lex} in_{lex}(Q'_2)\succ_{lex} in_{lex}(Q'_3)$ by replacing $Q_1,Q_2,Q_3$ with a linear combination of the three polynomials. We call such a basis $(Q'_1,Q'_2,Q'_3)$ normalized basis of $\Lambda$.

Finally, given a basis $\{x_0,x_1,\dots,x_5\}$ of $V$, we define the distinguished flag $F_0\subset F_1\subset F_2\subset F_3\subset F_4\subset \mathbb{P}(V)$ as follows:
\begin{equation*}
\begin{aligned}
&F_0: x_1=x_2=x_3=x_4=x_5=0 , \\
&F_1: x_2=x_3=x_4=x_5=0 , \\
&F_2: x_3=x_4=x_5=0 , \\
&F_3: x_4=x_5=0 , \\
&F_4: x_5=0 . \\
\end{aligned}
\end{equation*}

\begin{theorem}\label{1-PS}
Suppose that $\Lambda$ is stable with respect to every one-parameter subgroup of the following numerical types:
\begin{align*}
&{\rm(1)} \ \rho_1=(1,1,1,1,1,-5). && {\rm(8)} \ \rho_8=(4,1,1,-2,-2,-2). \\
&{\rm(2)} \ \rho_2=(1,1,1,1,-2,-2). && {\rm(9)} \ \rho_9=(3,1,1,-1,-1,-3). \\
&{\rm(3)} \ \rho_3=(1,1,1,-1,-1,-1). && {\rm(10)} \ \rho_{10}=(2,1,0,0,-1,-2). \\
&{\rm(4)} \ \rho_4=(2,2,-1,-1,-1,-1). && {\rm(11)} \ \rho_{11}=(5,5,-1,-1,-1,-7). \\
&{\rm(5)} \ \rho_5=(5,-1,-1,-1,-1,-1). && {\rm(12)} \ \rho_{12}=(1,1,0,0,-1,-1). \\
&{\rm(6)} \ \rho_6=(2,2,2,-1,-1,-4). && {\rm(13)} \ \rho_{13}=(5,3,1,-1,-3,-5). \\
&{\rm(7)} \ \rho_7=(7,1,1,1,-5,-5).
\end{align*}
Then $\Lambda$ is stable.
\end{theorem}

\begin{rmk}
By the above Theorem, $\Lambda$ is stable with respect to a fixed torus $T$ if and only if it is stable with respect to all one-parameter subgroups in $T$ of the numerical types $\{\rho_i\}_{i=1}^{13}$.
\end{rmk}

Fix a net $\Lambda$ which is $\rho_i$-stable for each $1\le i\le 13$. By the Hilbert-Mumford numerical criterion, to prove that $\Lambda$ is stable, we must show that $\Lambda$ is stable with respect to an arbitrary 1-PS $\chi:\mathbb{C}^*\to SL(6)$. Without loss of generality, we may assume that $\chi$ is normalized, acting diagonally on the basis $\{x_0,\cdots , x_5\}$ with weights $(a,b,c,d,e,f)$. To prove the theorem, we must exhibit a Pl\"{u}cker coordinates that does not vanish on $\Lambda$ with positive $\chi$-weight. More explicitly, if $(Q_1,Q_2,Q_3)$ is a normalized basis of $\Lambda$, we must exhibit non-zero quadratic monomials $m_1,m_2,m_3$ in the variables $\{x_0,\cdots, x_5\}$ which appear with non-zero coefficient in $Q_1\wedge Q_2\wedge Q_3$ with $\sum_{k=1}^3 \omega_{\chi}(m_k)>0$. We begin with a preparatory lemma.

\begin{lemma}\label{lemma1}
If a net $\Lambda$ is $\rho_i$-stable for each $1\le i\le 5$,
then for a basis $\{x_0,\cdots , x_5\}$ of $V$, the normalized basis $(Q_1,Q_2,Q_3)$ of $\Lambda$ satisfies the following:

$(1)$ $Q_1,Q_2,Q_3\notin (x_5)$.

$(2)$ $(Q_2,Q_3)\nsubseteq (x_4,x_5)$ and $Q_3\notin (x_4,x_5)^2$.

$(3)$ $(Q_1,Q_2,Q_3)\nsubseteq (x_3,x_4,x_5)$ and either $(Q_2,Q_3)\nsubseteq (x_3,x_4,x_5)$ or $Q_3\notin (x_3,x_4,x_5)^2$.

$(4)$ $(Q_2,Q_3)\nsubseteq (x_2,x_3,x_4,x_5)^2$ and either $(Q_1,Q_2,Q_3)\nsubseteq (x_2,x_3,x_4,x_5)$ or $Q_3\notin (x_2,x_3,x_4,x_5)^2$.

$(5)$ $(Q_1,Q_2,Q_3)\nsubseteq (x_1,x_2,x_3,x_4,x_5)$ or $(Q_2,Q_3)\nsubseteq (x_1,x_2,x_3,x_4,x_5)^2$.

(i.e. $in_{lex}(Q_1)=x_0^2$ or $in_{lex}(Q_1),in_{lex}(Q_2)\in (x_0)$.)
\end{lemma}

\begin{proof}
$(1),(2),(3),(4),(5)$ follows immediately from $\rho_1, \rho_2, \rho_3, \rho_4, \rho_5$-stability of $\Lambda$, respectively.
\end{proof}

Let $M_i=in_{lex}(Q_i)$ for $i=1,2,3$. We can now begin the proof of Theorem~\ref{1-PS}.

\begin{proof}[Proof of Theorem~\ref{1-PS}]
We consider separately the following three cases:
\begin{equation*}
\begin{aligned}
&\rm{I}. \ && F_0 \ \text{is not in the base locus of} \ \Lambda ; \\
&\rm{II}. \ && F_0 \ \text{is in the base locus of} \ \Lambda \ \text{but} \ F_1 \ \text{is not}; \\
&\rm{III}. \ && F_1 \ \text{is in the base locus of} \ \Lambda .
\end{aligned}
\end{equation*}

$\odot$\textbf{Case I: $F_0$ is not a base point.} We have $M_1=x_0^2$.

$\bullet$\textbf{Case I.1:} $Q_2$ has a term $m_2\geqslant x_2^2$.
Then $Q_3$ has a term $m_3\geqslant x_4^2$, $M_3\geqslant x_3x_5$ and $M_2\geqslant x_1x_5$ by Lemma 0.1(1),(2) and (4), respectively. Suppose $\Lambda$ is not $\chi$-stable, then
\begin{small}\begin{equation*}
\omega_{\chi}(M_1)+ \left(\begin{array}{rl}
\omega_{\chi}(M_2) \\
\omega_{\chi}(m_2)
\end{array}\right)
+ \left(\begin{array}{rl}
\omega_{\chi}(M_3) \\
\omega_{\chi}(m_3)
\end{array}\right)
\le 0
\Longleftrightarrow 2a + \left(\begin{array}{rl}
b+f \\
2c
\end{array}\right)
+ \left(\begin{array}{rl}
d+f \\
2e
\end{array}\right)
\le 0.
\end{equation*}\end{small}

Then $a=b=c=d=e=f=0$. So we get a contradiction.
\vskip 2mm

$\bullet$\textbf{Case I.2:} $Q_2$ has no term $\geqslant x_2^2$(i.e. $M_2<x_2^2$).
Then $Q_3$ has a term $m_3\geqslant x_4^2$, $Q_2$ has a term $m_2\geqslant x_3^2$, $M_3\geqslant x_2x_5$ and $M_2\geqslant x_1x_5$ by Lemma 0.1(1),(2),(3) and (4), respectively. Suppose $\Lambda$ is not $\chi$-stable, then
\begin{small}\begin{equation*}
\omega_{\chi}(M_1)+ \left(\begin{array}{rl}
\omega_{\chi}(M_2) \\
\omega_{\chi}(m_2)
\end{array}\right)
+ \left(\begin{array}{rl}
\omega_{\chi}(M_3) \\
\omega_{\chi}(m_3)
\end{array}\right)
\le 0
\Longleftrightarrow 2a + \left(\begin{array}{rl}
b+f \\
2d
\end{array}\right)
+ \left(\begin{array}{rl}
c+f \\
2e
\end{array}\right)
\le 0.
\end{equation*}\end{small}

Then $\chi=(2k,2k,2k,-k,-k,-4k)$. This contradicts to $\rho_6$-stability.
\vskip 3mm

$\odot$\textbf{Case II: $F_0$ is a base point but $F_1$ is not in the base locus.}
By Lemma 0.1, we have the following conditions:

$(1)$ $Q_3$ has a term $\geqslant x_4^2$.

$(2)$ $Q_2$ has a term $\geqslant x_3^2$ and $M_3\geqslant x_3x_5$.

$(3)$ Either $Q_2$ has a term $\geqslant x_2^2$ or $M_3\geqslant x_2x_5$.

$(4)$ $Q_1$ has a term $\geqslant x_1^2$.

$(5)$ $M_1, M_2\in \{x_0x_1,x_0x_2,x_0x_3,x_0x_4,x_0x_5 \}$.
\vskip 1mm

$\bullet$\textbf{Case II.1:} $Q_2$ contains a term $m_2\geqslant x_2^2$. Then

$(1)$ $Q_3$ has a term $\geqslant x_4^2$.

$(2)$ $M_3\geqslant x_3x_5$.

$(3)$ $Q_2$ has a term $\geqslant x_2^2$.

$(4)$ $Q_1$ has a term $\geqslant x_1^2$.

$(5)$ $M_1, M_2\in \{x_0x_1,x_0x_2,x_0x_3,x_0x_4,x_0x_5 \}$.
\vskip 2mm

$\cdot$\textbf{Case II.1.a:} $M_1=x_0x_4$ and $M_2=x_0x_5$.

By $\rho_7=(7,1,1,1,-5,-5)$-stability, $Q_3$ has a term $\geqslant x_3^2$. Suppose $\Lambda$ is not $\chi$-stable, then
\begin{equation*}
\left(\begin{array}{rl}
a+e \\
2b
\end{array}\right)
+\left(\begin{array}{rl}
a+f \\
2c
\end{array}\right)
+ 2d \le 0.
\end{equation*}

Then $\chi=(4k,k,k,-2k,-2k,-2k)$. This contradicts to $\rho_8$-stability.
\vskip 2mm

$\cdot$\textbf{Case II.1.b:} $M_1=x_0x_3$ and $M_2=x_0x_5$.

By $\rho_8=(4,1,1,-2,-2,-2)$-stability, $Q_3$ has a term $\geqslant x_2x_5$. Suppose $\Lambda$ is not $\chi$-stable, then
\begin{equation*}
\left(\begin{array}{rl}
a+d \\
2b
\end{array}\right)
+ \left(\begin{array}{rl}
a+f \\
2c
\end{array}\right)
+ \left(\begin{array}{rl}
c+f \\
2e
\end{array}\right)
\le 0
\end{equation*}

Then $\chi=(2k,k,0,0,-k,-2k)$. This contradicts to $\rho_{10}$-stability.
\vskip 2mm

$\cdot$\textbf{Case II.1.c:} $M_1=x_0x_3$ and $M_2=x_0x_4$.

By $\rho_8=(4,1,1,-2,-2,-2)$-stability, $Q_3$ has a term $\geqslant x_2x_5$. Suppose $\Lambda$ is not $\chi$-stable, then
\begin{equation*}
\left(\begin{array}{rl}
a+d \\
2b
\end{array}\right)
+ \left(\begin{array}{rl}
a+e \\
2c
\end{array}\right)
+ \left(\begin{array}{rl}
c+f \\
2e
\end{array}\right)
\le 0
\end{equation*}

Then $a=b=c=d=e=f=0$. So we get a contradiction.
\vskip 2mm

$\cdot$\textbf{Case II.1.d:} $M_1=x_0x_2$ and $M_2=x_0x_5$.

Suppose $\Lambda$ is not $\chi$-stable, then
\begin{equation*}
\left(\begin{array}{rl}
a+c \\
2b
\end{array}\right)
+ \left(\begin{array}{rl}
a+f \\
2c
\end{array}\right)
+ \left(\begin{array}{rl}
d+f \\
2e
\end{array}\right)
\le 0
\end{equation*}

Then $\chi=(2k,k,0,0,-k,-2k)$. This contradicts to $\rho_{10}$-stability.
\vskip 2mm

$\cdot$\textbf{Case II.1.e:} $M_1=x_0x_2$ and $M_2=x_0x_4$.

Suppose $\Lambda$ is not $\chi$-stable, then
\begin{equation*}
\left(\begin{array}{rl}
a+c \\
2b
\end{array}\right)
+ \left(\begin{array}{rl}
a+e \\
2c
\end{array}\right)
+ \left(\begin{array}{rl}
d+f \\
2e
\end{array}\right)
\le 0
\end{equation*}

Then $\chi=(k,k,k,k,-2k,-2k)$. This contradicts to $\rho_2$-stability.
\vskip 2mm

$\cdot$\textbf{Case II.1.f:} $M_1=x_0x_2$ and $M_2=x_0x_3$.

Suppose $\Lambda$ is not $\chi$-stable, then
\begin{equation*}
\left(\begin{array}{rl}
a+c \\
2b
\end{array}\right)
+ \left(\begin{array}{rl}
a+d \\
2c
\end{array}\right)
+ \left(\begin{array}{rl}
d+f \\
2e
\end{array}\right)
\le 0
\end{equation*}

Then $a=b=c=d=e=f=0$. We get a contradiction.
\vskip 2mm

$\cdot$\textbf{Case II.1.g:} $M_1=x_0x_1$ and $M_2=x_0x_5$.

Suppose $\Lambda$ is not $\chi$-stable, then
\begin{equation*}
(a+b) + \left(\begin{array}{rl}
a+f \\
2c
\end{array}\right)
+ \left(\begin{array}{rl}
d+f \\
2e
\end{array}\right)
\le 0
\end{equation*}

Then $a=b=c=d=e=f=0$. We get a contradiction.
\vskip 2mm

$\cdot$\textbf{Case II.1.h:} $M_1=x_0x_1$ and $M_2=x_0x_4$.

Suppose $\Lambda$ is not $\chi$-stable, then
\begin{equation*}
(a+b) + \left(\begin{array}{rl}
a+e \\
2c
\end{array}\right)
+ \left(\begin{array}{rl}
d+f \\
2e
\end{array}\right)
\le 0
\end{equation*}

Then $a=b=c=d=e=f=0$. We get a contradiction.
\vskip 2mm

$\cdot$\textbf{Case II.1.i:} $M_1=x_0x_1$ and $M_2=x_0x_3$.

Suppose $\Lambda$ is not $\chi$-stable, then
\begin{equation*}
(a+b) + \left(\begin{array}{rl}
a+d \\
2c
\end{array}\right)
+ \left(\begin{array}{rl}
d+f \\
2e
\end{array}\right)
\le 0
\end{equation*}

Then $a=b=c=d=e=f=0$. We get a contradiction.
\vskip 2mm

$\cdot$\textbf{Case II.1.j:} $M_1=x_0x_1$ and $M_2=x_0x_2$.

Suppose $\Lambda$ is not $\chi$-stable, then
\begin{equation*}
(a+b) + (a+c) + \left(\begin{array}{rl}
d+f \\
2e
\end{array}\right)
\le 0
\end{equation*}

Then $a=b=c=d=e=f=0$. We get a contradiction.
\vskip 2mm

$\bullet$\textbf{Case II.2:} $Q_2$ has no term $m_2\geqslant x_2^2$ (i.e. $(Q_2,Q_3)\subset (x_3,x_4,x_5)$). Then

$(1)$ $Q_3$ has a term $\geqslant x_4^2$.

$(2)$ $Q_2$ has a term $\geqslant x_3^2$.

$(3)$ $M_3\geqslant x_2x_5$.

$(4)$ $Q_1$ has a term $\geqslant x_1^2$.

$(5)$ $M_1, M_2\in \{x_0x_1,x_0x_2,x_0x_3,x_0x_4,x_0x_5 \}$.
\vskip 2mm

$\cdot$\textbf{Case II.2.a:} $M_1=x_0x_4$ and $M_2=x_0x_5$.

By $\rho_7=(7,1,1,1,-5,-5)$-stability, $Q_3$ has a term $\geqslant x_3^2$. And by $\rho_9=(3,1,1,-1,-1,-3)$-stability, $Q_3$ has a term $\geqslant x_2x_4$. Suppose $\Lambda$ is not $\chi$-stable, then
\begin{equation*}
\left(\begin{array}{rl}
a+e \\
2b
\end{array}\right)
+ \left(\begin{array}{rl}
a+f \\
2d
\end{array}\right)
+ \left(\begin{array}{rl}
c+e \\
2d
\end{array}\right)
\le 0
\end{equation*}

Then $a=b=c=d=e=f=0$. We get a contradiction.
\vskip 2mm

$\cdot$\textbf{Case II.2.b:} $M_1=x_0x_3$ and $M_2=x_0x_5$.

By $\rho_9=(3,1,1,-1,-1,-3)$-stability, $Q_3$ has a term $\geqslant x_2x_4$. Suppose $\Lambda$ is not $\chi$-stable, then
\begin{equation*}
\left(\begin{array}{rl}
a+d \\
2b
\end{array}\right)
+ \left(\begin{array}{rl}
a+f \\
2d
\end{array}\right)
+ (c+e) \le 0
\end{equation*}

Then $a=b=c=d=e=f=0$. We get a contradiction.
\vskip 2mm

$\cdot$\textbf{Case II.2.c:} $M_1=x_0x_3$ and $M_2=x_0x_4$.

Suppose $\Lambda$ is not $\chi$-stable, then
\begin{equation*}
\left(\begin{array}{rl}
a+d \\
2b
\end{array}\right)
+ \left(\begin{array}{rl}
a+e \\
2d
\end{array}\right)
+ \left(\begin{array}{rl}
c+f \\
2e
\end{array}\right)
\le 0
\end{equation*}

Then $a=b=c=d=e=f=0$. We get a contradiction.
\vskip 2mm

$\cdot$\textbf{Case II.2.d:} $M_1=x_0x_2$ and $M_2=x_0x_5$.

If $Q_2$ has a term $\geqslant x_1x_3$, then $\Lambda$ is $\chi$-stable.
Indeed, if $\Lambda$ is not $\chi$-stable, then
\begin{equation*}
\left(\begin{array}{rl}
a+c \\
2b
\end{array}\right)
+ \left(\begin{array}{rl}
a+f \\
b+d
\end{array}\right)
+ \left(\begin{array}{rl}
c+f \\
2e
\end{array}\right)
\le 0
\end{equation*}

Then $a=b=c=d=e=f=0$. We get a contradiction.
\vskip 2mm

Now assume $Q_2$ has no term $\geqslant x_1x_3$. Then by $\rho_{10}=(2,1,0,0,-1,-2)$-stability, $Q_3$ has a term $\geqslant x_1x_5$ or $\geqslant x_3x_4$. If $Q_3$ has a term $\geqslant x_1x_5$, then $\Lambda$ is $\chi$-stable.
Indeed, if $\Lambda$ is not $\chi$-stable, then
\begin{equation*}
\left(\begin{array}{rl}
a+c \\
2b
\end{array}\right)
+ \left(\begin{array}{rl}
a+f \\
2d
\end{array}\right)
+ \left(\begin{array}{rl}
b+f \\
2e
\end{array}\right)
\le 0
\end{equation*}

Then $\chi=(2k,2k,2k,-k,-k,-4k)$. This contradicts to $\rho_6$-stability.
\vskip 2mm

If $Q_3$ has a term $\geqslant x_3x_4$, then $\Lambda$ is $\chi$-stable.
Indeed, if $\Lambda$ is not $\chi$-stable, then
\begin{equation*}
\left(\begin{array}{rl}
a+c \\
2b
\end{array}\right)
+ \left(\begin{array}{rl}
a+f \\
2d
\end{array}\right)
+ \left(\begin{array}{rl}
c+f \\
d+e
\end{array}\right)
\le 0
\end{equation*}

Then $\chi=(2k,2k,2k,-k,-k,-4k)$. This contradicts to $\rho_6$-stability.
\vskip 2mm

$\cdot$\textbf{Case II.2.e:} $M_1=x_0x_2$ and $M_2=x_0x_4$.

Suppose $\Lambda$ is not $\chi$-stable, then
\begin{equation*}
\left(\begin{array}{rl}
a+c \\
2b
\end{array}\right)
+ \left(\begin{array}{rl}
a+e \\
2d
\end{array}\right)
+ \left(\begin{array}{rl}
c+f \\
2e
\end{array}\right)
\le 0
\end{equation*}

Then $a=b=c=d=e=f=0$. We get a contradiction.
\vskip 2mm

$\cdot$\textbf{Case II.2.f:} $M_1=x_0x_2$ and $M_2=x_0x_3$.

Suppose $\Lambda$ is not $\chi$-stable, then
\begin{equation*}
\left(\begin{array}{rl}
a+c \\
2b
\end{array}\right)
+ (a+d)
+ \left(\begin{array}{rl}
c+f \\
2e
\end{array}\right)
\le 0
\end{equation*}

Then $a=b=c=d=e=f=0$. We get a contradiction.
\vskip 2mm

$\cdot$\textbf{Case II.2.g:} $M_1=x_0x_1$ and $M_2=x_0x_5$.

Suppose $\Lambda$ is not $\chi$-stable, then
\begin{equation*}
(a+b)
+ \left(\begin{array}{rl}
a+f \\
2d
\end{array}\right)
+ \left(\begin{array}{rl}
c+f \\
2e
\end{array}\right)
\le 0
\end{equation*}

Then $\chi=(2k,2k,2k,-k,-k,-4k)$. This contradicts to $\rho_6$-stability.
\vskip 2mm

$\cdot$\textbf{Case II.2.h:} $M_1=x_0x_1$ and $M_2=x_0x_4$.

Suppose $\Lambda$ is not $\chi$-stable, then
\begin{equation*}
(a+b)
+ \left(\begin{array}{rl}
a+e \\
2d
\end{array}\right)
+ \left(\begin{array}{rl}
c+f \\
2e
\end{array}\right)
\le 0
\end{equation*}

Then $a=b=c=d=e=f=0$. We get a contradiction.
\vskip 2mm

$\cdot$\textbf{Case II.2.i:} $M_1=x_0x_1$ and $M_2=x_0x_3$.

Suppose $\Lambda$ is not $\chi$-stable, then
\begin{equation*}
(a+b)+(a+d)
+ \left(\begin{array}{rl}
c+f \\
2e
\end{array}\right)
\le 0
\end{equation*}

Then $a=b=c=d=e=f=0$. We get a contradiction.
\vskip 3mm

$\odot$\textbf{Case III: $F_1$ is in the base locus.}
By Lemma 0.1, we have the following conditions:

$(1)$ $Q_3$ has a term $\geqslant x_4^2$.

$(2)$ $Q_2$ has a term $\geqslant x_3^2$.

$(3)$ $Q_1$ has a term $\geqslant x_2^2$ and either $Q_2$ has a term $\geqslant x_2^2$ or $M_3\geqslant x_2x_5$.

$(4)$ $Q_1$ has no term $\geqslant x_1^2$(i.e. $Q_1,Q_2,Q_3\in (x_2,x_3,x_4,x_5)$) and $M_3\geqslant x_1x_5$.

$(5)$ $M_1, M_2\in \{x_0x_2,x_0x_3,x_0x_4,x_0x_5 \}$.
\vskip 1mm

$\bullet$\textbf{Case III.1:} $Q_2$ has a term $\geqslant x_2^2$. Then

$(1)$ $Q_3$ has a term $\geqslant x_4^2$.

$(2)$ $\times$

$(3)$ $Q_1$ has a term $\geqslant x_2^2$ and $Q_2$ has a term $\geqslant x_2^2$.

$(4)$ $Q_1,Q_2,Q_3\in (x_2,x_3,x_4,x_5)$ and $M_3\geqslant x_1x_5$.

$(5)$ $M_1, M_2\in \{x_0x_2,x_0x_3,x_0x_4,x_0x_5 \}$.
\vskip 2mm

$\cdot$\textbf{Case III.1.a:} $M_1=x_0x_4$ and $M_2=x_0x_5$.

By $\rho_7=(7,1,1,1,-5,-5)$-stability, $Q_3$ has a term $\geqslant x_3^2$.
If $Q_2$ has a term $\geqslant x_1x_4$, then $\Lambda$ is $\chi$-stable.
Indeed, if $\Lambda$ is not $\chi$-stable, then
\begin{equation*}
\left(\begin{array}{rl}
a+e \\
2c
\end{array}\right)
+ \left(\begin{array}{rl}
a+f \\
2c \\
b+e
\end{array}\right)
+\left(\begin{array}{rl}
b+f \\
2d
\end{array}\right)
\le 0
\end{equation*}

Then $\chi=(k,k,0,0,-k,-k)$. This contradicts to $\rho_{12}$-stability.
\vskip 2mm

Now assume $Q_2$ has no term $\geqslant x_1x_4$, then by $\rho_{11}=(5,5,-1,-1,-1,-7)$-stability, $Q_3$ has a term $\geqslant x_1x_4$. Suppose $\Lambda$ is not $\chi$-stable, then
\begin{equation*}
\left(\begin{array}{rl}
a+e \\
2c
\end{array}\right)
+ \left(\begin{array}{rl}
a+f \\
2c
\end{array}\right)
+\left(\begin{array}{rl}
b+e \\
2d
\end{array}\right)
\le 0
\end{equation*}

Then $\chi=(k,k,0,0,-k,-k)$. This contradicts to $\rho_{12}$-stability.
\vskip 2mm

$\cdot$\textbf{Case III.1.b:} $M_1=x_0x_3$ and $M_2=x_0x_5$.

Suppose $\Lambda$ is not $\chi$-stable, then
\begin{equation*}
\left(\begin{array}{rl}
a+d \\
2c
\end{array}\right)
+ \left(\begin{array}{rl}
a+f \\
2c
\end{array}\right)
+\left(\begin{array}{rl}
b+f \\
2e
\end{array}\right)
\le 0
\end{equation*}

Then $\chi=(5k,5k,-k,-k,-k,-7k)$. This contradicts to $\rho_{11}$-stability.
\vskip 2mm

$\cdot$\textbf{Case III.1.c:} $M_1=x_0x_3$ and $M_2=x_0x_4$.

Suppose $\Lambda$ is not $\chi$-stable, then
\begin{equation*}
\left(\begin{array}{rl}
a+d \\
2c
\end{array}\right)
+ \left(\begin{array}{rl}
a+e \\
2c
\end{array}\right)
+\left(\begin{array}{rl}
b+f \\
2e
\end{array}\right)
\le 0
\end{equation*}

Then $a=b=c=d=e=f=0$. We get a contradiction.
\vskip 2mm

$\cdot$\textbf{Case III.1.d:} $M_1=x_0x_2$ and $M_2=x_0x_5$.

Suppose $\Lambda$ is not $\chi$-stable, then
\begin{equation*}
(a+c)
+ \left(\begin{array}{rl}
a+f \\
2c
\end{array}\right)
+\left(\begin{array}{rl}
b+f \\
2e
\end{array}\right)
\le 0
\end{equation*}

Then $\chi=(5k,5k,-k,-k,-k,-7k)$. This contradicts to $\rho_{11}$-stability.
\vskip 2mm

$\cdot$\textbf{Case III.1.e:} $M_1=x_0x_2$ and $M_2=x_0x_4$.

Suppose $\Lambda$ is not $\chi$-stable, then
\begin{equation*}
(a+c)
+ \left(\begin{array}{rl}
a+e \\
2c
\end{array}\right)
+\left(\begin{array}{rl}
b+f \\
2e
\end{array}\right)
\le 0
\end{equation*}

Then $a=b=c=d=e=f=0$. We get a contradiction.
\vskip 2mm

$\cdot$\textbf{Case III.1.f:} $M_1=x_0x_2$ and $M_2=x_0x_3$.

Suppose $\Lambda$ is not $\chi$-stable, then
\begin{equation*}
(a+c)
+ \left(\begin{array}{rl}
a+d \\
2c
\end{array}\right)
+\left(\begin{array}{rl}
b+f \\
2e
\end{array}\right)
\le 0
\end{equation*}

Then $a=b=c=d=e=f=0$. We get a contradiction.
\vskip 2mm

$\bullet$\textbf{Case III.2:} $Q_2$ has no term $\geqslant x_2^2$.(i.e. $(Q_2,Q_3)\subset (x_3,x_4,x_5)$).
Then

$(1)$ $Q_3$ has a term $\geqslant x_4^2$.

$(2)$ $Q_2$ has a term $\geqslant x_3^2$.

$(3)$ $Q_1$ has a term $\geqslant x_2^2$.

$(4)$ $Q_1,Q_2,Q_3\in (x_2,x_3,x_4,x_5)$ and $M_3\geqslant x_1x_5$.

$(5)$ $M_1, M_2\in \{x_0x_2,x_0x_3,x_0x_4,x_0x_5 \}$.
\vskip 2mm

$\cdot$\textbf{Case III.2.a:} $M_1=x_0x_4$ and $M_2=x_0x_5$.

By $\rho_7=(7,1,1,1,-5,-5)$-stability, $Q_3$ has a term $\geqslant x_3^2$. And by $\rho_9=(3,1,1,-1,-1,-3)$-stability, $Q_3$ has a term $\geqslant x_2x_4$.
If $Q_1$ has a term $\geqslant x_1x_2$, then $\Lambda$ is $\chi$-stable.
Indeed, if $\Lambda$ is not $\chi$-stable, then
\begin{equation*}
\left(\begin{array}{rl}
a+e \\
b+c
\end{array}\right)
+\left(\begin{array}{rl}
a+f \\
2d
\end{array}\right)
+\left(\begin{array}{rl}
b+f \\
c+e \\
2d
\end{array}\right)
\le 0
\end{equation*}

Then $\chi=(5k,5k,-k,-k,-k,-7k)$. This contradicts to $\rho_{11}$-stability.
\vskip 2mm

Now assume that $Q_1$ has no term $\geqslant x_1x_2$. If $Q_3$ has a term $\geqslant x_1x_4$, then $\Lambda$ is $\chi$-stable.
Indeed, if $\Lambda$ is not $\chi$-stable, then
\begin{equation*}
\left(\begin{array}{rl}
a+e \\
2c
\end{array}\right)
+\left(\begin{array}{rl}
a+f \\
2d
\end{array}\right)
+\left(\begin{array}{rl}
b+e \\
2d
\end{array}\right)
\le 0
\end{equation*}

Then $\chi=(k,k,0,0,-k,-k)$. This contradicts to $\rho_{12}$-stability.
\vskip 2mm

If $Q_3$ has no term $\geqslant x_1x_4$, then by $\rho_{11}=(5,5,-1,-1,-1,-7)$-stability, $Q_2$ has a term $\geqslant x_1x_4$.

If $Q_3$ also has a term $\geqslant x_2x_3$. Then $\Lambda$ is $\chi$-stable.
Indeed, if $\Lambda$ is not $\chi$-stable, then
\begin{equation*}
\left(\begin{array}{rl}
a+e \\
2c
\end{array}\right)
+\left(\begin{array}{rl}
a+f \\
2d \\
b+e
\end{array}\right)
+\left(\begin{array}{rl}
b+f \\
c+d
\end{array}\right)
\le 0
\end{equation*}

Then $\chi=(k,k,0,0,-k,-k)$. This contradicts to $\rho_{12}$-stability.
\vskip 2mm

Now assume $Q_3$ has no term $\geqslant x_2x_3$. Then by $\rho_{13}=(5,3,1,-1,-3,-5)$-stability, $Q_2$ has a term $\geqslant x_1x_3$. Suppose $\Lambda$ is not $\chi$-stable, then
\begin{equation*}
\left(\begin{array}{rl}
a+e \\
2c
\end{array}\right)
+\left(\begin{array}{rl}
a+f \\
b+d
\end{array}\right)
+\left(\begin{array}{rl}
b+f \\
2d \\
c+e
\end{array}\right)
\le 0
\end{equation*}

Then $a=b=c=d=e=f=0$. We get a contradiction.
\vskip 2mm

$\cdot$\textbf{Case III.2.b:} $M_1=x_0x_3$ and $M_2=x_0x_5$.

By $\rho_9=(3,1,1,-1,-1,-3)$-stability, $Q_3$ has a term $\geqslant x_2x_4$.
If $Q_1$ has a term $\geqslant x_1x_2$, then $\Lambda$ is $\chi$-stable.
Indeed, if $\Lambda$ is not $\chi$-stable, then
\begin{equation*}
\left(\begin{array}{rl}
a+d \\
b+c
\end{array}\right)
+\left(\begin{array}{rl}
a+f \\
2d
\end{array}\right)
+\left(\begin{array}{rl}
b+f \\
c+e
\end{array}\right)
\le 0
\end{equation*}

Then $\chi=(5k,5k,-k,-k,-k,-7k)$. This contradicts to $\rho_{11}$-stability.
\vskip 2mm

Now assume $Q_1$ has no term $\geqslant x_1x_2$. If $Q_3$ has a term $\geqslant x_1x_4$, then $\Lambda$ is $\chi$-stable.
Indeed, if $\Lambda$ is not $\chi$-stable, then
\begin{equation*}
\left(\begin{array}{rl}
a+d \\
2c
\end{array}\right)
+\left(\begin{array}{rl}
a+f \\
2d
\end{array}\right)
+ (b+e) \le 0
\end{equation*}

Then $a=b=c=d=e=f=0$. We get a contradiction.
\vskip 2mm

Assume now $Q_3$ has no term $\geqslant x_1x_4$. By $\rho_{11}=(5,5,-1,-1,-1,-7)$-stability, $Q_2$ has a term $\geqslant x_1x_4$. Suppose $\Lambda$ is not $\chi$-stable, then
\begin{equation*}
\left(\begin{array}{rl}
a+d \\
2c
\end{array}\right)
+\left(\begin{array}{rl}
a+f \\
2d \\
b+e
\end{array}\right)
+\left(\begin{array}{rl}
b+f \\
c+e
\end{array}\right)
\le 0
\end{equation*}

Then $a=b=c=d=e=f=0$. We get a contradiction.
\vskip 2mm

$\cdot$\textbf{Case III.2.c:} $M_1=x_0x_3$ and $M_2=x_0x_4$.

Suppose $\Lambda$ is not $\chi$-stable, then
\begin{equation*}
\left(\begin{array}{rl}
a+d \\
2c
\end{array}\right)
+\left(\begin{array}{rl}
a+e \\
2d
\end{array}\right)
+\left(\begin{array}{rl}
b+f \\
2e
\end{array}\right)
\le 0
\end{equation*}

Then $a=b=c=d=e=f=0$. We get a contradiction.
\vskip 2mm

$\cdot$\textbf{Case III.2.d:} $M_1=x_0x_2$ and $M_2=x_0x_5$.

If $Q_2$ has a term $\geqslant x_1x_4$. Then $\Lambda$ is $\chi$-stable.
Indeed, if $\Lambda$ is not $\chi$-stable, then
\begin{equation*}
(a+c)
+\left(\begin{array}{rl}
a+f \\
2d \\
b+e
\end{array}\right)
+\left(\begin{array}{rl}
b+f \\
2e
\end{array}\right)
\le 0
\end{equation*}

Then $a=b=c=d=e=f=0$. We get a contradiction.
\vskip 2mm

Now assume $Q_2$ has no term $\geqslant x_1x_4$. By $\rho_{11}=(5,5,-1,-1,-1,-7)$-stability, $Q_3$ has a term $\geqslant x_1x_4$. Suppose $\Lambda$ is not $\chi$-stable, then
\begin{equation*}
(a+c)
+\left(\begin{array}{rl}
a+f \\
2d
\end{array}\right)
+ (b+e) \le 0
\end{equation*}

Then $a=b=c=d=e=f=0$. We get a contradiction.
\vskip 2mm

$\cdot$\textbf{Case III.2.e:} $M_1=x_0x_2$ and $M_2=x_0x_4$.

Suppose $\Lambda$ is not $\chi$-stable, then
\begin{equation*}
(a+c)
+\left(\begin{array}{rl}
a+e \\
2d
\end{array}\right)
+\left(\begin{array}{rl}
b+f \\
2e
\end{array}\right)
\le 0
\end{equation*}

Then $a=b=c=d=e=f=0$. We get a contradiction.
\vskip 2mm

$\cdot$\textbf{Case III.2.f:} $M_1=x_0x_2$ and $M_2=x_0x_3$.

Suppose $\Lambda$ is not $\chi$-stable, then
\begin{equation*}
(a+c)+(a+d)
+\left(\begin{array}{rl}
b+f \\
2e
\end{array}\right)
\le 0
\end{equation*}

Then $a=b=c=d=e=f=0$. We get a contradiction.
\end{proof}

\begin{lemma}\label{1to5}
A net $\Lambda$ is not stable with respect to one of $\{\rho_i \}_{i=1}^5$ if and only if it satisfies one of the following conditions with respect to a distinguished flag $F_0\subset F_1\subset F_2\subset F_3\subset F_4\subset \mathbb{P}^5$.
\vskip 2mm

$(1)$ $\rho_1=(1,1,1,1,1,-5):$

$\ \ (a)$ An element of $\Lambda$ contains $F_4$.
\vskip 1mm

$(2)$ $\rho_2=(1,1,1,1,-2,-2):$

$\ \ (a)$ A pencil of $\Lambda$ contains $F_3$, or

$\ \ (b)$ An element of $\Lambda$ is singular along $F_3$.
\vskip 1mm

$(3)$ $\rho_3=(1,1,1,-1,-1,-1):$

$\ \ (a)$ $\Lambda$ contains $F_2$, or

$\ \ (b)$ A pencil of $\Lambda$ contains $F_2$, and an element of the pencil is singular along $F_2$.
\vskip 1mm

$(4)$ $\rho_4=(2,2,-1,-1,-1,-1):$

$\ \ (a)$ $\Lambda$ contains $F_1$, and an element of $\Lambda$ is singular along $F_1$, or

$\ \ (b)$ A pencil of $\Lambda$ is singular along $F_1$.
\vskip 1mm

$(5)$ $\rho_5=(5,-1,-1,-1,-1,-1):$

$\ \ (a)$ $\Lambda$ contains $F_0$, and a pencil of $\Lambda$ is singular at $F_0$.

\end{lemma}

\begin{proof}
In case $(5)$, $(4,-2,-2)$ and $(-2,-2,-2)$ are all triples of initial $\rho_5$-weights with non-positive sum. Any quadric of weight $4$ contains $F_0$ and any quadric of weight $-2$ is singular at $F_0$. So the net with initial $\rho_5$-weights $(4,-2,-2)$ or $(-2,-2,-2)$ has a base point at $F_0$ and contains a pencil of quadrics singular at $F_0$. The proofs of cases $(1)-(4)$ are similar.
\end{proof}

On the basis of this partial analysis, we may already conclude the important fact that a stable net has a pure two-dimensional intersection, and hence defines a connected surface with local complete intersection singularities.

\begin{cor}\label{ci}
If a net of quadrics in $\mathbb{P}^5$ is stable, then the corresponding intersection is connected and purely two-dimensional.
\end{cor}

\begin{proof}
The proof is basically same as the proof of Corollary 2.8. in \cite{FS} and so we omit details.

Connectedness follows from Fulton-Hansen connectedness theorem \cite{FH}.
%
%
%
Suppose the intersection fails to be purely 2-dimensional. Then either a pencil of quadrics in the net contains a hyperplane, in which case the net is not $\rho_1$-stable, or we may choose a basis $\{Q_1,Q_2,Q_3\}$ of the net such that $Y:=Q_1\cap Q_2$ is a quartic 3-fold and there is an irreducible component $Y'\subset Y$ of degree at most 3 which is contained in $Q_3$. The net is not $\rho_2$-stable (resp., not $\rho_1$-stable) if the degree of $Y'$ is $1$ (resp., $2$).

If $degY'=3$, then $Y'$ is a rational normal scroll. If $Y'$ is smooth, then the net is projectively equivalent to $(x_0x_3-x_1x_2,x_0x_5-x_1x_4,x_2x_5-x_3x_4)$ and is not $\rho_4$-stable. If $Y'$ is singular, then it is either $S_{0,0,3}$ or $S_{0,1,2}$. If $Y'=S_{0,0,3}$, then it is singular along a line. So we must have a pencil of quadrics singular along $F_1$ after the coordinate change. Such a net is not $\rho_4$-stable. If $Y'$ is a cone over $S_{1,2}$ and $F_0$ denotes the vertex of the cone, then we must have a pencil of quadrics singular at $F_0$ and the net contains $F_0$. Such a net is not $\rho_5$-stable.
\end{proof}

\begin{rmk}
The Segre 3-fold $(x_0x_3-x_1x_2,x_0x_5-x_1x_4,x_2x_5-x_3x_4)$ is strictly semistable. It is stabilized by a certain 1-PS acting diagonally with respect to the distinguished basis $\{x_0,x_1,\cdots ,x_5\}$. Indeed, it is stabilized by $(4,2,1,-1,-2,-4)$. By the Kempf-Morisson criterion (Proposition 2.4 in \cite{AFS}), it therefore suffices to check that it is semistable with respect to 1-PS's acting diagonally with respect to this basis.

Now, let $\lambda =(a,b,c,d,e,f)$ be a 1-PS with $a+b+c+d+e+f=0$. Suppose it is unstable with respect to $\lambda$.
Then $\omega_{\lambda}(x_0x_3-x_1x_2)+\omega_{\lambda}(x_0x_5-x_1x_4)+\omega_{\lambda}(x_2x_5-x_3x_4)=max\{a+d,b+c\}+max\{a+f,b+e\}+max\{c+f,d+e\}<0$.
In particular, $0>(a+d)+(b+e)+(c+f)=a+b+c+d+e+f=0$. It is a contradiction.
%
%
%
%
%
%
%
%
%
%
%
%
%
%
%
%
Thus the Segre 3-fold is semistable. But it is not $\rho_4$-stable and so is strictly semistable.
\end{rmk}

\begin{lemma}\label{6to13}
Suppose $\Lambda$ is not $\rho_i$-stable for $i\in \{6,\cdots ,13\}$ but is $\rho_j$-stable for $1\le j\le i-1$. Let $m_1, m_2, m_3$ be the initial monomials of $\Lambda$ with respect to $\rho_i$. Then $(\omega_{\rho_i}(m_1), \omega_{\rho_i}(m_2), \omega_{\rho_i}(m_3))$ must be one of the following triples:
\vskip 2mm

$(6)$ $\rho_6=(2,2,2,-1,-1,-4):$

$\ \ \ \bullet$ $(4,-2,-2)$
\vskip 1mm

$(7)$ $\rho_7=(7,1,1,1,-5,-5):$

$\ \ \ \bullet$ $(2,2,-4)$
\vskip 1mm

$(8)$ $\rho_8=(4,1,1,-2,-2,-2):$

$\ \ \ \bullet$ $(2,2,-4)$
\vskip 1mm

$(9)$ $\rho_9=(3,1,1,-1,-1,-3):$

$\ \ \ \bullet$ $(2,0,-2)$
\vskip 1mm

$(10)$ $\rho_{10}=(2,1,0,0,-1,-2):$

$\ \ \ \bullet$ $(2,0,-2)$
\vskip 1mm

$(11)$ $\rho_{11}=(5,5,-1,-1,-1,-7):$

$\ \ \ \bullet$ $(4,-2,-2)$
\vskip 1mm

$(12)$ $\rho_{12}=(1,1,0,0,-1,-1):$

$\ \ \ \bullet$ $(0,0,0)$
\vskip 1mm

$(13)$ $\rho_{13}=(5,3,1,-1,-3,-5):$

$\ \ \ \bullet$ $(2,0,-2)$

\end{lemma}

\begin{proof}
Consider $\rho_i=(a,b,c,d,e,f)$ for $6\le i\le 13$ and suppose $\omega_1\ge \omega_2\ge \omega_3$ is the triple of $\rho_i$-initial weights of a $\rho_i$-nonstable net $\Lambda$. We can translate Lemma~\ref{lemma1} into the following conditions:
\vskip 1mm

$(C1)$ $\omega_3\ge 2e$.

$(C2)$ $\omega_2\ge 2d$ and $\omega_3\ge d+f$.

$(C3)$ $\omega_1\ge 2c$. Moreover, if $\omega_2 < 2c$, then $\omega_3\ge c+f$.

$(C4)$ $\omega_2\ge b+f$. Moreover, if $\omega_1 < 2b$, then $\omega_3\ge b+f$.

$(C5)$ If $\omega_1\neq 2a$, then $\omega_1\ge a+e$ and $\omega_2\ge a+f$.
\vskip 1mm

Now for each $\rho_i$, we list all triples of $\rho_i$-initial weights with non-positive sum and satisfy $(C1)-(C5)$. We will do only case (13), by far the most involved, and the remaining cases can be obtained more easily.

The set of possible $\rho_{13}=(5,3,1,-1,-3,-5)$-weights of quadratic monomials is
\[\{10,8,6,4,2,0,-2,-4,-6,-8,-10\}.\]
Suppose $\omega_1\ge \omega_2\ge \omega_3$ are initial $\rho_{13}$-weights of $\rho_{13}$-nonstable net $\Lambda$ and $\Lambda$ is $\rho_i$-stable for $1\le i\le 12$. By $(C2)$, $\omega_2\ge -2$ and $\omega_3\ge -6$. If $\omega_1=10$, then there is no triple with non-positive sum.
Suppose $\omega_1<10$, then $\omega_1\ge 2$ and $\omega_2\ge 0$ by $(C5)$. The triples with non-positive sum satisfying these conditions are

$\ \ \bullet$ $(6,0,-6)$, which violates $(C3)$;

$\ \ \bullet$ $(4,2,-6)$, which violates $(C4)$;

$\ \ \bullet$ $(4,0,-6)$, which violates $(C4)$;

$\ \ \bullet$ $(4,0,-4)$, which violates $(C4)$;

$\ \ \bullet$ $(2,2,-4)$, which violates $(C4)$;

$\ \ \bullet$ $(2,2,-6)$, which violates $(C4)$;

$\ \ \bullet$ $(2,0,-4)$, which violates $(C4)$;

$\ \ \bullet$ $(2,0,-6)$, which violates $(C4)$;

$\ \ \bullet$ $(2,0,-2)$.
\end{proof}

\section{A net of quadrics and the associated discriminant}

To a net $\Lambda=(Q_1, Q_2, Q_3)$ of quadrics in $\mathbb{P}^5$, we associate the sextic polynomial $\Delta(\Lambda)=det(xQ_1+yQ_2+zQ_3)$ which is called the discriminant of $\Lambda$.

While the GIT stability of a pencil of quadrics can be read off the GIT stability of the associated discriminant locus (cf. \cite{AM}), the GIT analysis for nets of quadrics turns out to be more involved. In particular, as following example shows, there is no natural correspondence between $SL(6)$-stability of a net and $SL(3)$-stability of the associated discriminant curve.

\begin{exm}\label{exm1}
Let $\Lambda=(2 x_0x_4+2 x_1x_3,2 x_0x_5+2 x_1x_4+2 x_2x_3,2 x_1x_5-2 x_2x_4)$. Then $\Lambda$ is strictly semistable. It is stabilized by a certain 1-PS acting diagonally with respect to the distinguished basis $\{x_0,x_1,\cdots ,x_5\}$. Indeed, it is stabilized by $(3,2,1,-1,-2,-3)$. By the Kempf-Morisson criterion (Proposition 2.4 in \cite{AFS}), it therefore suffices to check that it is semistable with respect to 1-PS's acting diagonally with respect to this basis. Suppose it is unstable with respect to a 1-PS $\lambda=(a,b,c,d,e,f)$ with $a+b+c+d+e+f=0$. Then $\omega_{\lambda}(2 x_0x_4+2 x_1x_3)+\omega_{\lambda}(2 x_0x_5+2 x_1x_4+2 x_2x_3)+\omega_{\lambda}(2 x_1x_5-2 x_2x_4)=max\{a+e,b+d\}+max\{a+f,b+e,c+d\}+max\{b+f,c+e\}<0$.
In particular, $0>(a+e)+(c+d)+(b+f)=a+b+c+d+e+f=0$. It is a contradiction.

Thus $\Lambda$ is semistable. But the discriminant sextic $\Delta(\Lambda)$ of $\Lambda$ is $-y^6$, which is unstable under the natural $SL(3)$-action on the space of plane sextic since $\omega_{\chi}(-y^6)=-6<0$ for $\chi=(2,-1,-1)$.
\end{exm}

\begin{prop}
If the discriminant of $\Lambda$ is identically zero then $\Lambda$ is unstable.
\end{prop}

\begin{proof}
Since a net $\Lambda$ of quadrics in $\mathbb{P}^5$ consists of singular quadrics, we have one of the followings (cf. Corollary 1 in \cite{W1}):

(1) Quadrics in $\Lambda$ have a common singular point; or

(2) Restricted to a common hyperplane, the quadrics in $\Lambda$ are singular along a line; or

(3) Restricted to a common 3-dimensional linear space, the quadrics in $\Lambda$ are singular along a plane.

In case (3), we can take coordinates $x_0, x_1,\cdots , x_5$ such that the common 3-dimensional linear space $T$ is $x_4=x_5=0$ and the plane in $T$ is defined by $x_3=0$. So the matrix with respect to any quadric in $\Lambda$ has a form :
\begin{equation*}
\begin{bmatrix}
0 & 0 & 0 & 0 & \ast & \ast \\
0 & 0 & 0 & 0 & \ast & \ast \\
0 & 0 & 0 & 0 & \ast & \ast \\
0 & 0 & 0 & \ast & \ast & \ast \\
\ast & \ast & \ast & \ast & \ast & \ast \\
\ast & \ast & \ast & \ast & \ast & \ast
\end{bmatrix}
.
\end{equation*}
(Here $0$ denote a zero entry and $\ast$ that no restriction is imposed.)

Then $\Lambda$ is $(3,3,3,-1,-4,-4)$-unstable.

Similarly, in cases (1) and (2), we may assume the matrix with respect to any quadric in $\Lambda$ has a form :
\begin{equation*}
\begin{bmatrix}
0 & 0 & 0 & 0 & 0 & 0 \\
0 & \ast & \ast & \ast & \ast & \ast \\
0 & \ast & \ast & \ast & \ast & \ast \\
0 & \ast & \ast & \ast & \ast & \ast \\
0 & \ast & \ast & \ast & \ast & \ast \\
0 & \ast & \ast & \ast & \ast & \ast
\end{bmatrix}
\ \text{and} \
\begin{bmatrix}
0 & 0 & 0 & 0 & 0 & \ast \\
0 & 0 & 0 & 0 & 0 & \ast \\
0 & 0 & \ast & \ast & \ast & \ast \\
0 & 0 & \ast & \ast & \ast & \ast \\
0 & 0 & \ast & \ast & \ast & \ast \\
\ast & \ast & \ast & \ast & \ast & \ast
\end{bmatrix}
,\ \text{respectively}.
\end{equation*}
(Here $0$ denote a zero entry and $\ast$ that no restriction is imposed.)

And then $\Lambda$ is $\rho_5$-unstable and $\rho_{11}$-unstable, respectively.
\end{proof}

\vskip 2mm
If a net $\Lambda$ defines a complete intersection surface $S$, we will use the net $\Lambda$ and the defining surface $S$ interchangeably. In particular, if a net is stable, then we will use the net and the defining surface interchangeably because a stable net defines a complete intersection surface by Corollary~\ref{ci},.

In this section, we use the instability results of the previous section to discuss some connections between a net of quadrics and the associated discriminant sextic curve.

\begin{theorem}\label{normal}
Suppose a net $\Lambda$ of quadrics in $\mathbb{P}^5$ defines a complete intersection normal surface $S$. If the discriminant $\Delta(S)$ of $S$ is stable then $\Lambda$ is stable.
\end{theorem}

\begin{proof}
Suppose a net $\Lambda$ of quadrics in $\mathbb{P}^5$ defines a complete intersection normal surface $S$. Note that $S$ is irreducible and reduced.

Now suppose $\Lambda$ is not stable. Then by Theorem~\ref{1-PS}, there exist a basis of $V$ such that $\Lambda$ is not $\rho_i$-stable for some $1\le i\le 13$. Then we will show that $\Delta(S)$ is not stable. The non-stability of $\Lambda$ for $\{\rho\}_{1=1}^{5}$ uses Lemma~\ref{1to5} and for $\{\rho\}_{1=6}^{13}$ uses Lemma~\ref{6to13}.

If $\Lambda$ is not $\rho_i$-stable for $i\in \{1, 2, 3\}$, then $S$ is not normal by Lemma ~\ref{1to5} (1)--(3).

\vskip 1mm
If $\Lambda$ is not $\rho_4$-stable, then either $\Lambda$ contains $F_1$ and an element of $\Lambda$ is singular along $F_1$ or a pencil of $\Lambda$ is singular along $F_1$ with respect to some distinguished flag $F_0\subset F_1\subset F_2\subset F_3\subset F_4\subset \mathbb{P}^5$ by Lemma ~\ref{1to5} (4). If $\Lambda$ contains $F_1$, and an element of $\Lambda$ is singular along $F_1$ then $S$ is not normal. Hence, it suffices to consider the case when a pencil of $\Lambda$ is singular along $F_1$. That is $(Q_2,Q_3)\subset (x_2,x_3,x_4,x_5)^2$, where $\{x_0,\cdots , x_5 \}$ a basis of $V$ which define the flag. Then $\Delta(S)$ is contained in a linear span of $\{x^i y^j z^k| i+j+k=6$ and $i\ge 2\}$. Thus $\Delta(S)$ is not stable (cf. 1.9 in \cite{Mumford}).

\vskip 1mm
If $\Lambda$ is not $\rho_i$-stable for $\{5, 7, 8, 9, 10, 13\}$, then $\Delta(S)$ is not stable by the same argument as the latter case of $\rho_4$.

\vskip 1mm
If $\Lambda$ is not $\rho_6$-stable, then by Lemma~\ref{6to13} (6), $(\omega_{\rho_6}(m_1)$,$\omega_{\rho_6}(m_2)$,$\omega_{\rho_6}(m_3))=(4,-2,-2)$ for some basis $\{x_0,\cdots , x_5 \}$ of $V$. Then the equations of $Q_i$ can be written as
\begin{equation*}
\begin{aligned}
&Q_1: q_1(x_0,\cdots , x_5)=0,\\
&Q_2: x_5 l_2(x_0,\cdots , x_5)+q_2(x_3, x_4)=0,\\
&Q_3: x_5 l_3(x_1,\cdots , x_5)+q_3(x_3, x_4)=0.
\end{aligned}
\end{equation*}
where $q_i$ are quadratic polynomials and $l_i$ are linear polynomials. Then $S$ is singular along $C: x_3=x_4=x_5=Q_1=0$ and so is not normal.

\vskip 1mm
Similarly, if $\Lambda$ is not $\rho_i$-stable for $i\in \{11, 12\}$, then $S$ is singular along $L: x_2=x_3=x_4=x_5=0$ for some basis $\{x_0,\cdots , x_5 \}$ of $V$ and so is not normal.
\end{proof}

\begin{theorem}\label{reduce}
Suppose a net $\Lambda$ of quadrics in $\mathbb{P}^5$ defines a complete intersection surface $S$ with the reduced discriminant $\Delta(S)$. If $\Delta(S)$ is stable then $\Lambda$ is stable.
\end{theorem}

\begin{proof}
Let a net $\Lambda$ of quadrics in $\mathbb{P}^5$ define a complete intersection surface $S$ with the reduced discriminant $\Delta(S)$.

Now suppose $\Lambda$ is not stable. Then by Theorem~\ref{1-PS} $\Lambda$ is not $\rho_i$-stable for some $1\le i\le 13$. Then we will show that $\Delta(S)$ is not stable. By the proof of Theorem~\ref{normal}, $\Delta(S)$ is not stable if $\Lambda$ is not $\rho_i$-stable for some $i\in\{5,7,8,9,10,13\}$ or a pencil of $\Lambda$ is singular along $F_1$ with respect to some distinguished flag $F_0\subset F_1\subset F_2\subset F_3\subset F_4\subset \mathbb{P}^5$, and so $\Lambda$ is not $\rho_4$-stable. Therefore we consider remaining cases.

\vskip 1mm
If $\Lambda$ is not $\rho_{12}=(1,1,0,0,-1,-1)$-stable, then by Lemma~\ref{6to13} (12), $(\omega_{\rho_{12}}(m_1)$,$\omega_{\rho_{12}}(m_2)$,$\omega_{\rho_{12}}(m_3))$$=(0,0,0)$ for some basis $\{x_0,\cdots , x_5 \}$ of $V$. Then the matrix $A_i$ with respect to each $Q_i$ has the form:
\begin{SMALL}\begin{equation*}
A_1=\begin{bmatrix}
0 & 0 & 0 & 0 & a_{04} & a_{05} \\
0 & 0 & 0 & 0 & a_{14} & a_{15} \\
0 & 0 & 2a_{22} & a_{23} & a_{24} & a_{25} \\
0 & 0 & a_{23} & 2a_{33} & a_{34} & a_{35} \\
a_{04} & a_{14} & a_{23} & a_{33} & 2a_{34} & a_{35} \\
a_{05} & a_{15} & a_{23} & a_{33} & a_{34} & 2a_{35}
\end{bmatrix}
, A_2=\begin{bmatrix}
0 & 0 & 0 & 0 & b_{04} & b_{05} \\
0 & 0 & 0 & 0 & b_{14} & b_{15} \\
0 & 0 & 2b_{22} & b_{23} & b_{24} & b_{25} \\
0 & 0 & b_{23} & 2b_{33} & b_{34} & b_{35} \\
b_{04} & b_{14} & b_{23} & b_{33} & 2b_{34} & b_{35} \\
b_{05} & b_{15} & b_{23} & b_{33} & b_{34} & 2b_{35}
\end{bmatrix}
,
\end{equation*}\end{SMALL}
\begin{SMALL}\begin{equation*}
A_3=\begin{bmatrix}
0 & 0 & 0 & 0 & c_{04} & c_{05} \\
0 & 0 & 0 & 0 & c_{14} & c_{15} \\
0 & 0 & 2c_{22} & c_{23} & c_{24} & c_{25} \\
0 & 0 & c_{23} & 2c_{33} & c_{34} & c_{35} \\
c_{04} & c_{14} & c_{23} & c_{33} & 2c_{34} & c_{35} \\
c_{05} & c_{15} & c_{23} & c_{33} & c_{34} & 2c_{35}
\end{bmatrix}
\end{equation*}\end{SMALL}
and so

\begin{equation*}
\begin{aligned}
&\Delta(S)=det(xA_1+yA_2+zA_3)\\
&=-q_1^2q_2
\end{aligned}
\end{equation*}
where $q_1, q_2$ are quadratic polynomials in $x, y, z$. Thus $\Delta(S)$ is not reduced.

\vskip 1mm
If $\Lambda$ is not $\rho_3$-stable, then either $\Lambda$ contains $F_2$ or a pencil of $\Lambda$ contains $F_2$, and an element of the pencil is singular along $F_2$ with respect to some distinguished flag $F_0\subset F_1\subset F_2\subset F_3\subset F_4\subset \mathbb{P}^5$ by Lemma ~\ref{1to5} (3). First consider the case when $\Lambda$ contains $F_2$. That is $(Q_1,Q_2,Q_3)\subset (x_3,x_4,x_5)$, where $\{x_0,\cdots , x_5 \}$ a basis of $V$ which define the flag. Then the matrix $A_i$ with respect to each $Q_i$ has the form:
\begin{SMALL}\begin{equation*}
A_1=\begin{bmatrix}
0 & 0 & 0 & a_{03} & a_{04} & a_{05} \\
0 & 0 & 0 & a_{13} & a_{14} & a_{15} \\
0 & 0 & 0 & a_{23} & a_{24} & a_{25} \\
a_{03} & a_{13} & a_{23} & 2a_{33} & a_{34} & a_{35} \\
a_{04} & a_{14} & a_{23} & a_{33} & 2a_{34} & a_{35} \\
a_{05} & a_{15} & a_{23} & a_{33} & a_{34} & 2a_{35}
\end{bmatrix}
, A_2=\begin{bmatrix}
0 & 0 & 0 & 0 & b_{04} & b_{05} \\
0 & 0 & 0 & b_{13} & b_{14} & b_{15} \\
0 & 0 & 0 & b_{23} & b_{24} & b_{25} \\
0 & b_{13} & b_{23} & 2b_{33} & b_{34} & b_{35} \\
b_{04} & b_{14} & b_{23} & b_{33} & 2b_{34} & b_{35} \\
b_{05} & b_{15} & b_{23} & b_{33} & b_{34} & 2b_{35}
\end{bmatrix}
,
\end{equation*}\end{SMALL}
\begin{SMALL}\begin{equation*}
A_3=\begin{bmatrix}
0 & 0 & 0 & 0 & 0 & c_{05} \\
0 & 0 & 0 & c_{13} & c_{14} & c_{15} \\
0 & 0 & 0 & c_{23} & c_{24} & c_{25} \\
0 & c_{13} & c_{23} & 2c_{33} & c_{34} & c_{35} \\
0 & c_{14} & c_{23} & c_{33} & 2c_{34} & c_{35} \\
c_{05} & c_{15} & c_{23} & c_{33} & c_{34} & 2c_{35}
\end{bmatrix}
\end{equation*}\end{SMALL}
and so

\begin{equation*}
\begin{aligned}
\Delta(S)&=det(xA_1+yA_2+zA_3)\\
&=det \left( \begin{bmatrix}
O & A \\
A^T & B \\
\end{bmatrix} \right)\\
&=-(det(A))^2
\end{aligned}
\end{equation*}
where
\begin{Small}\begin{equation*}
O=\begin{bmatrix}
0 & 0 & 0 \\
0 & 0 & 0 \\
0 & 0 & 0
\end{bmatrix},
A=\begin{bmatrix}
a_{03}x & a_{04}x+b_{04}y & a_{05}x+b_{05}y+c_{05}z \\
a_{13}x+b_{13}y+c_{13}z & a_{14}x+b_{14}y+c_{14}z & a_{15}x+b_{15}y+c_{15}z \\
a_{23}x+b_{23}y+c_{23}z & a_{24}x+b_{24}y+c_{24}z & a_{25}x+b_{25}y+c_{25}z \\
\end{bmatrix}.
\end{equation*}\end{Small}
Thus $\Delta(S)$ is not reduced.

\vskip 1mm
Now consider the case when a pencil of $\Lambda$ contains $F_2$, and an element of the pencil is singular along $F_2$. That is $(Q_2,Q_3)\subset (x_3,x_4,x_5)$ and $Q_3\in (x_3,x_4,x_5)^2$, where $\{x_0,\cdots , x_5 \}$ a basis of $V$ which define the flag. Then $\Delta(S)$ is contained in a linear span of $x^6$, $x^5y$, $x^5z$, $x^4y^2$, $x^4yz$, $x^4z^2$, $x^3y^3$, $x^3y^2z$, $x^3yz^2$, $x^3z^3$, $x^2y^4$, $x^2y^3z$, $x^2y^2z^2$, $xy^5$, $xy^4z$, $y^6$. Thus $\Delta(S)$ is not stable (cf. 1.9 in \cite{Mumford}).

\vskip 1mm
Finally, if $\Lambda$ is not $\rho_i$-stable for $i\in \{1, 2, 4, 6, 11\}$, then $\Delta(S)$ is not stable by the same argument as the latter case of $\rho_3$.
\end{proof}

Let $S$ be a K3 surface of degree $8$ in $\mathbb{P}^5$, given by the complete intersection of three quadrics $Q_1, Q_2, Q_3$. Associated to $S$ is a K3 surface $S'$ which is a double cover of $\mathbb{P}^2$ ramified over a sextic. Let $\Lambda$ be the net of quadrics spanned by $Q_0, Q_1, Q_2$, and $\psi: X \longrightarrow \mathbb{P}^2$ the double cover of $\mathbb{P}^2$ branched along $\Delta(\Lambda)$. Then $S'$ is also a K3 surface. And there is a dominant rational map $\phi: \mathcal{N}^{ss}//SL(6)\dashrightarrow \mathcal{C}_6^{ss}//SL(3)$ where $\mathcal{N}$ the space of nets of quadrics in $\mathbb{P}^5$ and $\mathcal{C}_6$ the space of plane sextic curves. The degree of this map is equal to the number of non-effective theta characteristics on a general sextic curve which is $2^9(2^{10}+1)$ (cf. \cite{Do} and \cite{Sh}). Our original motivation in this paper is to compare the moduli space of algebraic K3 surfaces with degree $8$ with the moduli space of K3 surfaces as a double cover of $\mathbb{P}^2$, ramified over a sextic curve. It is well known that if $S$ is nonsingular then $\Delta(S)$ is stable \cite{Be2}.

\begin{theorem}\label{simple}
Suppose a net $\Lambda$ of quadrics in $\mathbb{P}^5$ defines a complete intersection surface $S$. If $S$ has simple singularities, then $\Delta(S)$ has simple singularities.
\end{theorem}

We begin with some preparation for the proof of Theorem~\ref{simple}. We refer \cite{W2} for it. And we also refer \cite{GLS} for the singularity theory.

We consider a net $\Lambda$ of quadrics given by the vanishing of
\[ F(\lambda,x)\equiv \sum_{i=1}^3 \sum_{j,k=0}^5 a_{jk}^i \lambda_i x_j x_k=x^T(\sum_{i=1}^3 \lambda_i A_i)x, \]
where the matrices $A_i=(a_{jk}^i)$ are symmetric. Define the \emph{total variety}
\[ V=\{(\lambda,x):F(\lambda,x)=0 \}, \]
the \emph{variety of base points}
\[ B=\{x:F(\lambda,x)=0 \text{ for all } \lambda \}, \]
and the \emph{discriminant}
\[ \Delta:=\Delta(\Lambda)=\{\lambda:det(\sum_{i=1}^3 \lambda_i A_i)=0 \}. \]

Write also for any $\lambda$, $Q_{\lambda}$ for the quadric $\{x: F(\lambda,x)=0\}$, and $Q_i$ for the quadric $x^T A_i x=0$.

For any variety $X$, we write $S(X)$ for the variety of its singular points: we first consider this as a point set. Thus $S(Q_{\lambda})$ is the vertex of the quadric $Q_{\lambda}$.

\begin{lemma}{\rm (\cite{W2},Lemma 1.1.)}
$(\lambda,x)\in S(V)$ if and only if $x\in S(Q_{\lambda})\cap B.$
\end{lemma}

\begin{lemma}{\rm (\cite{W2},Lemma 1.2.)}
$x\in S(B)$ if and only if there exists $\lambda \in \Delta$ with $x\in S(Q_{\lambda}).$
\end{lemma}

Combining these lemmas, we see that the image of $S(V)$ under the projection on $\mathbb{P}^5$ defined by $x$ is precisely $S(B)$. This projection often gives a bijection $S(V)\to S(B)$. Indeed given $x\in B$, we again consider the tangent planes $x^T A_i y=0$. In general these are independent, spanning a $3$-dimensional vector space. We have $x\in S(B)$ when they are dependent: call $x$ \emph{tame} if they span a $2$-dimensional space. In this case there is a unique linear relation (up to scalar multiples), hence a unique $\lambda \in \mathbb{P}^2$ with $x\in S(Q_{\lambda})$ and so $(\lambda,x)\in S(V)$. Therefore $S(V)\to S(B)$ is bijective if all points in $S(B)$ are tame.

\begin{lemma}
For a stable net $\Lambda$, all points of $S(B)$ are tame.
\end{lemma}

\begin{proof}
Suppose that $x\in S(B)$ is not tame. Then there are $\lambda_1\neq \lambda_2$ such that $x$ is on the vertex of $Q_{\lambda_1}, Q_{\lambda_2}$and therefore $\Lambda$ is not $\rho_5$-stable by Lemma ~\ref{1to5} (5).
\end{proof}

Say that two isolated hypersurface singularities are \emph{the same type} if either they are analytically equivalent or they can be reduced by analytic equivalence to hypersurfaces defined by two functions
\[ f(z_1,\cdots ,z_k)=0 \text{ in } \mathbb{C}^k \]
\[ f(z_1,\cdots ,z_k)+\sum_{i=1}^t z_{k+i}^2=0 \text{ in } \mathbb{C}^{k+t}. \]

Here of course, $\sum_{i=1}^t z_{k+i}^2$ could be replaced by any nonsingular quadratic form.

\begin{prop}{\rm (\cite{W2},Proposition 1.3.)}
Let $(\lambda,x)\in S(V)$ where $x$ is tame in $S(B)$. Then the two singular points have the same type.
\end{prop}

So for proving Theorem~\ref{simple} it is enough to show that if $V$ has simple singularities, then $\Delta$ has simple singularities because that if $S$ has simple singularities then $\Lambda$ is stable (\cite{OSS},\cite{Odaka1},\cite{Odaka2} and \cite{LiTian}).

\vskip 2mm
Now take coordinates such that the point $\lambda$ under investigation is at $\Lambda_1=(1,0,0)$ and $Q_1$(of corank ($k+1$)) has equation
\[ \sum_{i=k+1}^5 x_i^2=0. \]
We can take affine coordinates by setting $\lambda_1=1$, but for now we retain all the $x_i$. We partition all the matrices into blocks, separating the first $(k+1)$ rows and columns from the remaining $(5-k)$; thus
\[ M=A_1+\lambda_2 A_2+\lambda_3 A_3=
\begin{bmatrix}
A & B^T \\
B & C
\end{bmatrix}. \]
Here $A,B$ and $C-I$ are homogeneous linear in the $\lambda_2,\lambda_3.$ So $C$ is invertible at, and hence near $\Lambda_1.$ We use the identity
\[ \begin{bmatrix}
A & B^T \\
B & C
\end{bmatrix}
\begin{bmatrix}
I & 0 \\
-C^{-1}B & I
\end{bmatrix}
=\begin{bmatrix}
A-B^TC^{-1}B & B^T \\
0 & C
\end{bmatrix} \]
to compute the determinant of $M$. We find
\[ \text{det}M=\text{det}C\text{ det}(A-B^TC^{-1}B).\]
so the equation of $\Delta$ can be written $0=\text{det}(A-U)$, where $U=B^TC^{-1}B.$

To study $V$, we partition the coordinate vector $x^T=(y^T,z^T)$ correspondingly. Then
\[ F(\lambda,x)=x^T M x=y^T A y+2y^T B^T z+z^T C z;\]
for fixed $y$ (and small $\lambda$) this is a nonsingular quadratic in $z$, with centre $x=-C^{-1}By.$ Setting $z=z'-C^{-1}By$, we have
\[ F(\lambda,x)=z'^T C z'+y^T A y-y^T B^T C^{-1} B y=z'^T C z'+y^T (A-U) y, \]
and the singularity of this has the same type as that of $y^T(A-U)y.$

In all, the singularity of $\Delta$ is given by $0=\text{det}(A-U)$, those of $V$ have the same type as those of $0=y^T(A-U)y$.

\begin{rmk}
While $A$ is homogeneous linear in $\lambda_2,\lambda_3$, all the terms in $U$ have order $\ge 2$. Indeed, the terms of degree $2$ in $U$ are given by $B^T B$. The tangent cone of $\Delta$ at $\Lambda_1$ is given by the terms of lowest degree in det$(A-U)$, i.e. by det$(A)$.
\end{rmk}

We can now begin the proof of Theorem~\ref{simple}.

\vskip 2mm
\begin{proof}[Proof of Theorem~\ref{simple}]
Consider the projection $S(V)\to \Delta$.
If $\Lambda_1$ is an isolated singularity of $\Delta$, the corank $(k+1)$ of $Q_1$ must satisfy $k\le 2$. The results for $k=0$ are contained in Theorem 1.4. of \cite{W2}.

Now let $k=1$. The net cut on the vertex $S(Q_1)$ (a projective line) is spanned by two forms, so can be reduced to one of the normal forms
\[ {\rm(1)} \ \lambda x^2+\mu y^2 \ \ \ {\rm(2)} \ 2\lambda xy+\mu x^2 \ \ \ {\rm(3)} \ 2\lambda xy \ \ \ {\rm(4)} \ \lambda x^2 \ \ \ {\rm(5)} \ 0, \]
with respective discriminants
\[ {\rm(1)} \ \lambda \mu \ \ \ {\rm(2)} \ -\lambda^2 \ \ \ {\rm(3)} \ -\lambda^2 \ \ \ {\rm(4)} \ 0 \ \ \ {\rm(5)} \ 0. \]

The corresponding singularities of $V$ (or $B$) are the base points of this system. For case (5) we get the whole line as a non-isolated singularity: this we will not discuss further. In case (2),(4) we get one point $Y$ $(x=0)$; in case (3) two points $X$ $(y=0)$ and $Y$, and in case (1) no singularities.

Since the above discriminant is the tangent cone to $\Delta$ at $\Lambda_1$, we have a singularity of type $A_1$ in case (1), a higher double point $A_n$ in cases (2) and (3) (cf. Theorem 3.2. and Theorem 3.3 in \cite{W2}).

In case (4), we can normalize
\[ A= \begin{bmatrix}
\lambda & 0 \\
0 & 0
\end{bmatrix};\]
we also write
\[ U= \begin{bmatrix}
u & v \\
v & w
\end{bmatrix}.\]
Then $\Delta$ is given (locally) by
\[ (u-\lambda)w=v^2\]
and we may consider $V$ as given (near $Y$) by
\[ (u-\lambda)x^2+2vx+w=0 \]
in affine coordinates $(y=1)$.

Here we must first look at the terms of degree $2$ in $w$, giving a homogeneous quadratic $q$ in $\lambda,\mu$. If $q$ has distinct factors then $V$ has a singularity of corank $1$ and $\Delta$ one belonging to the $D$-series. If $q$ has a repeated root which is not $\lambda$, then $\Delta$ has a singularity of type $D_m$ for some $m$. If $q$ is a multiple of $\lambda^2$, $\Delta$ still has a triple point (of higher type) and $V$ a corank $2$ singularity. But the singularity of $V$ is a simple (actually $D_m$) only if the coefficient of $\mu^2$ in $v$ does not vanish. And in this case $\Delta$ has a singularity of type $E_6$.

Finally, let $k=2$. The vertex of the quadric $Q_1$ is a (projective) plane, and the net cuts a pencil of conics (defined by $Q_2$ and $Q_3$) in this plane. The corresponding points of $S(V)$ are the intersections of these conics: we wish these to be isolated. We may thus have a singular pencil of type $(x^2, y^2)$ or one of the five types of nonsingular pencil.

The tangent cone to $\Delta$ at $\Lambda_1$ is given by the discriminant of the pencil. If this has three distinct factors, $\Lambda_1$ has type $D_4$.

If the discriminant of the pencil has one repeated point and one other, $\Lambda_1$ has type $D_n$ for some $n$. We can take the pencil in the form
\[\mu(x_1^2+2\alpha x_0x_1)+\nu(2x_0x_2). \]
i.e. \[ A= \begin{bmatrix}
0 & \alpha \mu & \nu \\
\alpha \mu & \mu & 0 \\
\nu & 0 & 0
\end{bmatrix}.\]
And we write
\[ U= \begin{bmatrix}
a & b & c \\
b & d & e \\
c & e & f
\end{bmatrix}.\]
Then $\Delta$ is given (locally) by
\begin{equation*}
\begin{aligned}
0=&-cd^2+2bce-ae^2-b^2f+adf+c^2\mu-2\alpha ce\mu-af\mu+2\alpha bf\mu \\
&+2cd\nu-2be\nu-\alpha^2f\mu^2-2c\mu\nu+2\alpha e\mu\nu-d\nu^2+\mu\nu^2.
\end{aligned}
\end{equation*}
Since the 3-jet of $\Delta$ is $\mu\nu^2$, $\Lambda_1$ has type $D_n$ for some $n$.

If the discriminant of the pencil has a threefold point, we take the pencil as
\[ \mu(x_1^2-x_0x_2)+\nu(x_0^2+2\alpha x_0x_1). \]
i.e. \[ A= \begin{bmatrix}
\nu & \alpha \nu & \frac{-\mu}{2} \\
\alpha \nu & \mu & 0 \\
\frac{-\mu}{2} & 0 & 0
\end{bmatrix}.\]
And we write
\[ U= \begin{bmatrix}
a & b & c \\
b & d & e \\
c & e & f
\end{bmatrix}.\]
Then the corresponding singular points on $V$ are $(x_0,x_1,x_2)=(0,0,1)$ and $(4\alpha^2,-2\alpha,1)$ (if $\alpha \neq 0$).

$\Delta$ is given (locally) by
\begin{equation*}
\begin{aligned}
0=&-c^2d+2bce-ae^2-b^2f+adf+c^2\mu-cd\mu+be\mu-af\mu-2\alpha ce\nu \\
&+e^2\nu+2\alpha bf\nu-df\nu+c\mu^2-\frac{1}{4}d\mu^2-\alpha e\mu\nu+f\mu\nu-\alpha^2f\nu^2+\frac{1}{4}\mu^3.
\end{aligned}
\end{equation*}
and we may consider $V$ as given by
\[ 0=(a-\nu)x_0^2+2(b-\alpha\nu)x_0x_1+2(c+\frac{\mu}{2})x_0+(d-\mu)x_1^2+2ex_1+f \]
in affine coordinates $(x_2=1)$.

Let $F=(a-\nu)x_0^2+2(b-\alpha\nu)x_0x_1+2(c+\frac{\mu}{2})x_0+(d-\mu)x_1^2+2ex_1+f$ and write
\[\begin{array}{l}
c=l\mu^2+m\mu\nu+n\nu^2+c_{\ge3}\\
e=i\mu^2+j\mu\nu+k\nu^2+e_{\ge3}\\
f=s\mu^2+t\mu\nu+r\nu^2+u\nu^3+f_{\ge3}
\end{array}\]

Suppose $\alpha\neq 0$. We consider the following cases:
\begin{equation*}
\begin{aligned}
&\rm{I}. \ && r\neq 0 ; \\
&\rm{II}. \ && r=0 \text{ and } t\neq 0; \\
&\rm{III}. \ && r=0 \text{ and } t=0.
\end{aligned}
\end{equation*}

$\odot$\textbf{Case I:} $r\neq 0$.
Then the coefficient of $\nu^4$ comes from $-\alpha^2f\nu^2$ and is nonzero, so $\Lambda_1$ has $E_6$ type.

$\odot$\textbf{Case II:} $r=0$ and $t\neq 0$.
Then the coefficient of $\mu\nu^3$ comes from $-\alpha e\mu\nu-\alpha^2f\nu^2$ and is $-\alpha k-\alpha^2t$. So $\Lambda_1$ has $E_7$ type if $k\neq -\alpha t$.
If $k=-\alpha t$, The coefficient of $\nu^5$ comes from $-2\alpha ce\nu-e^2\nu-\alpha^2f\nu^2$ and is $-2\alpha nk+k^2-\alpha^2u$. So $\Lambda_1$ has $E_8$ type if $-2\alpha nk+k^2-\alpha^2u \neq 0$.

Now let $k=-\alpha t$ and $-2\alpha nk+k^2-\alpha^2u=0$. Then $\Lambda_1$ is not simple. So we observe the singularity of $V$.

There exists an automorphism $\varphi$ of $\mathbb{C}[[x_0,x_1,\mu,\nu]]$ such that $\varphi(F)=x_0^2+x_1^2+G(\mu,\nu)$ with $G\in <\mu,\nu>^3$. Transform $F$ by
\begin{equation*}
T=
\left\{\begin{array}{rl}
\begin{bmatrix}
0 & -2\sqrt{-s} & -t & 0 \\
0 & 0 & 0 & 1 \\
\frac{1}{\sqrt{s}} & -\frac{1}{\sqrt{-s}} & 0 & 0 \\
0 & 0 & 1 & 0
\end{bmatrix} \text{ if } s\neq 0\\
\begin{bmatrix}
1 & -\sqrt{-1} & -t & 0 \\
0 & 0 & 0 & 1 \\
1 & \sqrt{-1} & 0 & 0 \\
0 & 0 & 1 & 0
\end{bmatrix} \text{ if } s=0
\end{array}\right.
\end{equation*}
Then we can assume that
\[ F=x_0^2+x_1^2+F_{\ge3}(\mu,\nu)+\sum_{i=0,1}x_iG_i(x_0,x_1,\mu,\nu) \]
with $G_i\in <x_0,x_1,\mu,\nu>^2$. The coordinate change $x_i\mapsto x_i-\frac{1}{2}G_i$ for $i=0,1$, $\mu\mapsto\mu$ and $\nu\mapsto\nu$ yields
\[ F=x_0^2+x_1^2+F_{\ge3}(\mu,\nu)+F_{\ge4}(\mu,\nu)+\sum_{i=0,1}x_iH_i(x_0,x_1,\mu,\nu) \]
with $H_i\in <x_0,x_1,\mu,\nu>^3$. Continuing with $H_i$ instead of $G_i$ in the same manner, the last sum will be of arbitrary high order, hence $0$ in limit.

The $3$-jet of $G(\mu,\nu)$ is
\[\begin{array}{l}
-t^2\mu^3+2\alpha t\mu^2\nu-2nt\mu^3+2k\mu^2\nu+u\mu^3\\
=(-t^2-2nt+u)\mu^3+(2\alpha t+2k)\mu^2\nu \\
=0 \ \ (\text{because } k=-\alpha t \text{ and } -2\alpha nk+k^2-\alpha^2u=0).
\end{array}\]
So the singular point of $V$ is not simple.

$\odot$\textbf{Case III:} $r=0$ and $t=0$.
Then the coefficient of $\mu\nu^3$ comes from $-\alpha e\mu\nu$ and is $-\alpha k$. So $\Lambda_1$ has $E_7$ type if $k \neq 0$.
If $k=0$, The coefficient of $\nu^5$ comes from $-\alpha^2f\nu^2$ and is $-\alpha^2u$. So $\Lambda_1$ has $E_8$ type if $u \neq 0$.

Now let $k=0$ and $u=0$. Then $\Lambda_1$ is not simple. So we observe the singularity of $V$.

Same way as above, there exists an automorphism $\varphi$ of $\mathbb{C}[[x_0,x_1,\mu,\nu]]$ such that $\varphi(F)=x_0^2+x_1^2+G(\mu,\nu)$ with $G\in <\mu,\nu>^3$ and the $3$-jet of $G(\mu,\nu)$ is zero. Thus the singular point of $V$ is not simple.

\vskip 1mm
Now suppose $\alpha=0$. We consider the following cases:
\begin{equation*}
\begin{aligned}
&\rm{I}. \ && r\neq 0 ; \\
&\rm{II}. \ && r=0 \text{ and } k\neq 0; \\
&\rm{III}. \ && r=0 \text{ and } k=0.
\end{aligned}
\end{equation*}

$\odot$\textbf{Case I:} $r\neq 0$.
Then the coefficient of $\mu\nu^3$ comes from $f\mu\nu$ and is nonzero, so $\Lambda_1$ has $E_7$ type.

$\odot$\textbf{Case II:} $r=0$ and $k\neq 0$.
Then the coefficient of $\nu^5$ comes from $e^2\nu$ and is $k^2$. So $\Lambda_1$ has $E_8$.

$\odot$\textbf{Case III:} $r=0$ and $k=0$.
Then $\Lambda_1$ is not simple. So we observe the singularity of $V$. Again there exists an automorphism $\varphi$ of $\mathbb{C}[[x_0,x_1,\mu,\nu]]$ such that $\varphi(F)=x_0^2+x_1^2+G(\mu,\nu)$ with $G\in <\mu,\nu>^3$. We will show $G\in <\mu,\nu^2>^3$, then the singular point of $V$ is not simple.

After transform $F$ by $T$ we can assume that
\[ F=x_0^2+x_1^2+F_{\ge3}(\mu,\nu)+\sum_{i=0,1}x_iG_i(x_0,x_1,\mu,\nu) \]
with $G_i\in <x_0,x_1,\mu,\nu>^2$. The coordinate change $x_i\mapsto x_i-\frac{1}{2}G_i$ for $i=0,1$, $\mu\mapsto\mu$ and $\nu\mapsto\nu$ yields
$$
\begin{aligned}
F=&\;x_0^2+x_1^2+F_{\ge3}(\mu,\nu)-x_0G_0-x_1G_1+\frac{1}{4}G_0^2+\frac{1}{4}G_1^2\\
&+(x_0-\frac{1}{2}G_0)\cdot G_0(x_0-\frac{1}{2}G_0,x_1-\frac{1}{2}G_1,\mu,\nu)\\
&+(x_1-\frac{1}{2}G_1)\cdot G_1(x_0-\frac{1}{2}G_0,x_1-\frac{1}{2}G_1,\mu,\nu)
\end{aligned}
$$
Let
\[\begin{array}{l}
G_0:=G_0(x_0,x_1,\mu,\nu)=g_1+x_0g_2+x_1g_3+q+x_0^2g_4+x_0x_1g_5+x_1^2g_6,\\
G_1:=G_1(x_0,x_1,\mu,\nu)=h_1+x_0h_2+x_1h_3+q'+x_0^2h_4+x_0x_1h_5+x_1^2h_6
\end{array}\]
with $g_1,h_1\in <\mu,\nu>^2,g_2,g_3,h_2,h_3\in <\mu,\nu>,q,q'$ is quadratic in $x_0,x_1$ and $g_4,g_5,g_6,h_4,h_5,h_6\in <x_0,x_1,\mu,\nu>$.

Then
\[\begin{array}{l}
(x_0-\frac{1}{2}G_0)\cdot G_0(x_0-\frac{1}{2}G_0,x_1-\frac{1}{2}G_1,\mu,\nu)\\
=x_0G_0-\frac{1}{2}G_0^2+\frac{1}{4}G_0^2g_2+\frac{1}{4}G_0G_1g_3+x_0g'+x_1g''
\end{array}\]
and
\[\begin{array}{l}
(x_1-\frac{1}{2}G_1)\cdot G_1(x_0-\frac{1}{2}G_0,x_1-\frac{1}{2}G_1,\mu,\nu)\\
=x_1G_1-\frac{1}{2}G_1^2+\frac{1}{4}G_0G_1h_2+\frac{1}{4}G_1^2h_3+x_0h'+x_1h''
\end{array}\]
with $g',g'',h',h''\in <x_0,x_1,\mu,\nu>^3$. So
\[\begin{aligned}
F=&\;x_0^2+x_1^2+F_{\ge3}-\frac{1}{4}G_0^2-\frac{1}{4}G_1^2+\frac{1}{4}G_0^2g_2+\frac{1}{4}G_0G_1g_3+\frac{1}{4}G_0G_1h_2\\
&+\frac{1}{4}G_1^2h_3+F_{\ge6}+x_0\tilde{G_0}+x_1\tilde{G_1}
\end{aligned}\]
with $F_{\ge6}\in <x_0,x_1,\mu,\nu>^6$ and $\tilde{G_0},\tilde{G_1}\in <x_0,x_1,\mu,\nu>^3$.

\vskip 1mm
Continuing this process, we can obtain
\[
F=x_0^2+x_1^2+F_{\ge3}-\frac{1}{4}g_1^2-\frac{1}{4}h_1^2+\frac{1}{4}g_1^2g_2+\frac{1}{4}g_1h_1g_3+\frac{1}{4}g_1h_1h_2+\frac{1}{4}h_1^2h_3+\tilde{F}_{\ge6},
\]
where $\tilde{F}_{\ge6}\in <x_0,x_1,\mu,\nu>^6$.

\vskip 1mm
Thus the $3$-jet comes from
$$
F_{\ge3},
$$
$\nu^4$ term and $\mu \nu^3$ term comes from
$$
F_{\ge3}-\frac{1}{4}g_1^2-\frac{1}{4}h_1^2
$$
and $\nu^5$ term comes from
$$
F_{\ge3}-\frac{1}{4}g_1^2-\frac{1}{4}h_1^2+\frac{1}{4}g_1^2g_2+\frac{1}{4}g_1h_1g_3+\frac{1}{4}g_1h_1h_2+\frac{1}{4}h_1^2h_3.
$$

\vskip 2mm
Now, recall that $V$ is given by
\[ F=(a-\nu)x_0^2+2b-x_0x_1+2(c+\frac{\mu}{2})x_0+(d-\mu)x_1^2+2ex_1+f \]
in affine coordinates $(x_2=1)$.
Transform $F$ by $T$ and consider $3$-jet, $\nu^4$ term, $\mu \nu^3$ term and $\nu^5$ term.

If $s\ne 0$, by using
\begin{equation*}
T=\begin{bmatrix}
0 & -2\sqrt{-s} & -t & 0 \\
0 & 0 & 0 & 1 \\
\frac{1}{\sqrt{s}} & -\frac{1}{\sqrt{-s}} & 0 & 0 \\
0 & 0 & 1 & 0
\end{bmatrix}
\end{equation*}
the $3$-jet of $F_{\ge3}$ is
$(-t^2-2nt+u)\mu^3$ and so $F_{\ge3}\in <\mu,\nu^2>^3$.

Also we can obtain the followings:

$x_0\nu^2$ term of $x_0g_1$ is $-\frac{1}{\sqrt{s}}x_0\nu^2$ and so $\nu^2$ term of $g_1$ is $-\frac{1}{\sqrt{s}}\nu^2$,

$x_1\nu^2$ term of $x_1h_1$ is $\frac{1}{\sqrt{-s}}x_1\nu^2$ and so $\nu^2$ term of $h_1$ is $\frac{1}{\sqrt{-s}}\nu^2$,

$x_0\mu \nu$ term of $x_0g_1$ is $\frac{j}{\sqrt{s}}x_0\mu \nu$ and so $\mu \nu$ term of $g_1$ is $\frac{j}{\sqrt{s}}\mu \nu$,

$x_1\mu \nu$ term of $x_1h_1$ is $-\frac{j}{\sqrt{-s}}x_1\mu \nu$ and so $\mu \nu$ term of $h_1$ is $-\frac{j}{\sqrt{-s}}\mu \nu$,

$x_0^2\nu$ term of $x_0^2g_2$ is $\frac{i}{3}x_0^2\nu$ and so $\nu$ term of $g_2$ is $\frac{i}{3}\nu$,

$x_1^2\nu$ term of $x_1^2h_3$ is $-\frac{i}{3}x_1^2\nu$ and so $\nu$ term of $h_3$ is $-\frac{i}{3}\nu$,

$x_0x_1\nu$ term of $x_0x_1(g_3+h_2)$ is $-\frac{2i}{3\sqrt{-1}}x_0x_1\nu$ and so $\nu$ term of $g_3+h_2$ is $-\frac{2i}{3\sqrt{-1}}\nu$.

Thus $\nu^4$ term, $\mu \nu^3$ term and $\nu^5$ term are $0$.
In all, $G\in <\mu,\nu^2>^3$.

In case $s=0$, by using
\begin{equation*}
T=\begin{bmatrix}
1 & -\sqrt{-1} & -t & 0 \\
0 & 0 & 0 & 1 \\
1 & \sqrt{-1} & 0 & 0 \\
0 & 0 & 1 & 0
\end{bmatrix},
\end{equation*}
we can similarly show $G\in <\mu,\nu^2>^3$.
\end{proof}

\medskip

{\em Acknowledgements}. This work will be contained in part of my Ph.D. thesis. I would like to thank my advisor Yongnam Lee, for his advice, encouragement and teaching and the author is grateful to Brendan Hassett and Yuji Odaka for useful comments. And this work was supported by Basic Science Program through the National Research Foundation of Korea funded by the Korea government(MSIP)(No.2013006431).

\bigskip

\begin{small}\end{small}


\begin{thebibliography}{PPS}

\bibitem{AFS}
J. Alper, M. Fedorchuk and D. Smyth, \emph{Finite Hilbert stability of (bi)canonical curves}, Invent. Math.
\textbf{191} (2013), 671--718.

\bibitem{AL}
D. Avritzer, H. Lange, \emph{Pencils of quadrics, binary forms and hyperelliptic curves}, Comm. in Algebra \textbf{28} (2000), 5541--5561.

\bibitem{AM}
D. Avritzer, R. Miranda, \emph{Stability of pencils of quadrics in $\mathbb{P}^4$}, Bol. Soc. Mat. Mexicana. (3)
\textbf{5} (1999), 281--300.

\bibitem{Be2}
A. Beauville, \emph{Vari\'{e}t\'{e}s de Prym et jacobiennes interm\'{e}diaires}, Ann. Sci. \'{E}cole Norm. Sup. \textbf{10} (1977), no 3, 309--391.

\bibitem{Be1}
A. Beauville, \emph{Determinantal hypersurfaces}, Mich. Math. J.
\textbf{48} (2000), 39--64.

\bibitem{Do}
I. V. Dolgachev, Classical algebraic geometry: A modern view, Cambridge Univ. Press, 2012.

\bibitem{FS}
 M. Fedorchuk, D. Smyth, \emph{Stability of genus five canonical curves},  A celebration of algebraic geometry,
Clay Math. Proc.
\textbf{18} (2013), 281--310.

\bibitem{FH}
W. Fulton, J. Hansen, \emph{A connectedness theorem for projective varieties, with applications to intersections and singularities of mappings}, Ann. of Math. (2)
\textbf{110} (1979), 159--166.

\bibitem{GLS}
G.-M. Greuel, C. Lossen, E. Shustin, Introduction to singularities and deformations,
Springer Monographs in Mathematics, Springer-Verlag, Berlin, 2007.

\bibitem{Hassett}
B. Hasset, Personal communication.

\bibitem{LiTian}
Z. Li, Z. Tian, \emph{Picard groups of moduli space of low degree K3 surfaces}, math.AG arXiv:1304.3219v1, 2013.

\bibitem{MM}
T. Mabuchi, S. Mukai, \emph{Stability and Einstein-K\"{a}hler metric of a quartic
del Pezzo surface}, Einstein metrics and Yang-Mills connections (Sanda, 1990),
Lecture Notes in Pure and Appl. Math. \textbf{145} (1993), 133--160.

\bibitem{Mukai}
S. Mukai, \emph{Symplectic srtucture of the moduli space of sheaves on an abelian or K3 surface}, Invent. Math. \textbf{77} (1984), no 1, 101--116.

\bibitem{Mumford}
D. Mumford, \emph{Stability of projective varieties}, L'Ens. Math. \textbf{23} (1977), 39--110.

\bibitem{GIT}
D. Mumford, J. Fogarty, F. Kirwan, Geometric Invariant Theory, 3rd ed, Ergebnisse der Mathematik und ihrer Grenzgebiete 34, Springer-Verlag, Berlin, 1994.

\bibitem{Odaka1}
Y. Odaka, \emph{The Calabi Conjecture and K-stability}, Int. Math. Res. Not. \textbf{2012} (2012),  2272--2288.

\bibitem{Odaka2}
Y. Odaka, \emph{A generalization of the Ross-Thomas slope theory}, Osaka J. Math. \textbf{50} (2013), no 1, 171--185.

\bibitem{OSS}
Y. Odaka, C. Spotti, S. Sun, \emph{Compact moduli spaces of del Pezzo surfaces and K\"{a}hler-Einstein metrics}, math.AG arXiv:1210.0858v2, 2012.

\bibitem{Sh}
J. Shah, \emph{A complete moduli space for K3 surfaces of degree 2}, Ann. of Math. \textbf{112} (1980),  485--510.

\bibitem{W1}
C.T.C. Wall, \emph{Nets of quadrics, and theta-characteristics of singular curves}, Phil. Trans. Roy. Soc. London Ser. A
\textbf{289} (1978), 229--269.

\bibitem{W2}
C.T.C. Wall, \emph{Singularities of nets of quadrics}, Compositio Math.
\textbf{42} (1981), 187--212.



\end{thebibliography}
\end{document}